\documentclass[12pt,a4paper,english,reqno]{amsart}
\usepackage{amsthm}
\usepackage{amssymb}
\usepackage{amscd}
\usepackage{amstext}
\usepackage[all]{xy}
\xyoption{2cell}
\UseTwocells

\usepackage{graphicx}
\theoremstyle{plain}
\newtheorem{thm}{Theorem}[section]
\newtheorem{prp}[thm]{Proposition}
\newtheorem{lem}[thm]{Lemma}
\newtheorem{cor}[thm]{Corollary}
\newtheorem{clm}{Claim}

\newtheorem*{thm-nn}{Theorem}
\newtheorem*{prp-nn}{Proposition}
\newtheorem*{lem-nn}{Lemma}
\newtheorem*{cor-nn}{Corollary}
\newtheorem*{clm-nn}{Claim}
\newtheorem*{cnj-nn}{Conjecture}
\newtheorem*{prb-nn}{Problem}

\theoremstyle{definition}
\newtheorem{dfn}[thm]{Definition}
\newtheorem{exm}[thm]{Example}

\newtheorem*{dfn-nn}{Definition}

\newtheorem{rmk}[thm]{Remark}
\newtheorem{ntn}[thm]{Notation}

\newcommand{\xyR}[1]{% 
\xydef@\xymatrixrowsep@{#1}}

\newcommand{\xyC}[1]{% 
\xydef@\xymatrixcolsep@{#1}}

\def\al{\alpha}
\def\be{\beta}
\def\ga{\gamma}
\def\de{\delta}
\def\ep{\varepsilon}
\def\ze{\zeta}
\def\et{\eta}
\def\th{\theta}

\def\ro{\rho}
\def\si{\sigma}

\def\ph{\phi}
\def\ps{\psi}

\def\De{\Delta}

\def\Ps{\Psi}
\def\Th{\Theta}

\def\Mod{\operatorname{Mod}}

\def\prj{\operatorname{prj}}

\def\Prj{\operatorname{Prj}}

\def\thick{\operatorname{thick}}%-- added on 2008-07-26 ---
\def\Kb{{\mathcal K}^{\text{\rm b}}}

\def\calA{{\mathcal A}}
\def\calB{{\mathcal B}}
\def\calC{{\mathcal C}}
\def\calD{{\mathcal D}}

\def\calK{{\mathcal K}}

\def\calT{{\mathcal T}}
\def\calU{{\mathcal U}}
\def\calV{{\mathcal V}}

\def\bbZ{{\mathbb Z}}

\def\op{^{\mathrm{op}}} % changed text --> mathrm

\def\inv{^{-1}}

\def\incl{\hookrightarrow}
\def\iso{\cong}
\def\ds{\oplus}
\def\ox{\otimes}

\def\Lox{\overset{\mathbf{L}}{\otimes}}

\def\udl{\underline}
\def\ovl{\overline}
\def\Ds{\bigoplus}

\def\dsm#1,#2..#3{\bigoplus_{{#1}={#2}}^{#3}}
\def\sm#1,#2..#3{\sum_{{#1}={#2}}^{#3}}

\def\id{1\kern-.25em{\text{{\rm l}}}} %1\!\!{\text{{\rm l}}}   is good for 12pt

\def\isoto{\ \raise.8ex\hbox{$^{\sim}$}\kern-.7em\hbox{$\to$}\ } 
\def\ya#1{\xrightarrow{#1}}
\def\blank{\operatorname{-}}

\def\bg{%
\family{cmr}\size{20}{12pt}\selectfont}

\def\bigzerou{%
\smash{\lower1.7ex\hbox{\bg 0}}}

\def\repr[#1;#2;#3;#4;#5]{
\left(
\begin{matrix}#1\\#2\end{matrix}
#3
\begin{matrix}#4\\#5\end{matrix}
\right)}
\def\bmat#1{\begin{bmatrix} #1 \end{bmatrix}}
\def\smat#1{\begin{smallmatrix} #1 \end{smallmatrix}}

\def\sic{\operatorname{sic}}

\def\kCat{\Bbbk\text{-}\mathbf{Cat}}

\def\kAb{\Bbbk\text{-}\mathbf{Ab}}
\def\kTri{\Bbbk\text{-}\mathbf{Tri}}
\def\kModCat{\k\text{-}\mathbf{ModCat}}
\def\kadd{\Bbbk\text{-}\mathbf{add}}

\def\coop{^{\mathrm{coop}}}

\def\bfj{\mathbf{j}}
\def\bfp{\mathbf{p}}
\def\bfB{\mathbf{B}}
\def\bfC{\mathbf{C}}
\def\bfD{\mathbf{D}}
\def\bfL{\mathbf{L}}

\def\Gr{\operatorname{Gr}}
\def\k{\Bbbk}
\def\Colax{\overleftarrow{\operatorname{Colax}}}

\def\com{\operatorname{com}}
\def\To{\Rightarrow}

\begin{document}

\title{Gluing derived equivalences together}

\author{Hideto Asashiba}

\begin{abstract}
The Grothendieck construction of a diagram $X$ of categories
can be seen as a process to construct a single category $\Gr(X)$
by gluing categories in the diagram together.
Here we formulate diagrams of categories as colax functors
from a small category $I$ to the 2-category $\kCat$ of small $\k$-categories
for a fixed commutative ring $\k$.
In our previous paper we defined derived equivalences of those colax functors.
Roughly speaking two colax functors $X, X' \colon I \to \kCat$ are
derived equivalent
if there is a derived equivalence from $X(i)$ to $X'(i)$ for all objects $i$ in $I$
satisfying some ``$I$-equivariance'' conditions.
In this paper we glue the derived equivalences
between $X(i)$ and $X'(i)$ together to obtain a derived equivalence between Grothendieck constructions $\Gr(X)$ and $\Gr(X')$,
which shows that
if colax functors are derived equivalent, then
so are their Grothendieck constructions.
This generalizes and well formulates
the fact that if two $\k$-categories with a $G$-action for a group $G$ are
``$G$-equivariantly'' derived equivalent, then
their orbit categories are derived equivalent.
As an easy application we see by a unified proof
that if two $\Bbbk$-algebras
$A$ and $A'$ are derived equivalent, then so are
the path categories $AQ$ and $A'Q$ for any quiver $Q$;
so are the incidence categories $AS$ and $A'S$ for
any poset $S$; and so are the monoid algebras
$AG$ and $A'G$ for any monoid $G$.
Also we will give examples of gluing of many smaller derived equivalences together
to have a larger derived equivalence.
\end{abstract}


\thanks{This work is partially supported by Grant-in-Aid for Scientific Research
(C) 21540036 from JSPS}

\email{shasash@ipc.shizuoka.ac.jp}
\address{Department of Mathematics, Faculty of Science, Shizuoka University,
836 Ohya, Suruga-ku, Shizuoka, 422-8529, Japan.}
\maketitle

\section{Introduction}

Under the preparations in \cite{Asa-a} we complete our project of the title in this paper.
We fix a small category $I$ and a commutative ring $\Bbbk$ and
denote by $\kCat$ (resp.\ $\kAb$, $\kTri$)
the 2-category of small $\Bbbk$-categories
(resp.\ abelian $\k$-categories, triangulated $\k$-categories).
For a $\k$-category $\calC$ a (right) $\calC$-{\em module} is a contravariant functor from
$\calC$ to the category $\Mod \k$ of $\k$-modules,
and we denote by $\Mod \calC$ (resp.\ $\Prj \calC$, $\prj \calC$) the category of $\calC$-modules
(resp.\ projective $\calC$-modules, finitely generated projective $\calC$-modules).
When we deal with derived equivalences, we usually assume that
$\k$ is a field because Keller's theorem in \cite{Ke1} or \cite{Ke2}
on derived equivalences of categories
requires that the $\k$-categories in consideration are $\k$-flat or $\k$-projective.

A $\k$-category $\calC$ with an action of a group $G$ have been well investigated
in connection with a so-called covering technique in representation theory of algebras
(see e.g., \cite{Gab}).
The orbit category $\calC/G$ and the canonical functor $\calC \to \calC/G$ 
are naturally constructed from these data, and one studied relationships
between $\Mod \calC$ and $\Mod \calC/G$.
We brought this point of view to the derived equivalence classification problem
of algebras in \cite{Asa97}, and a main tool obtained there
was fully used in the derived equivalence classifications
in \cite{Asa99, Asa02}.  The main tool was extended in \cite{Asa11} in
the following form:

\begin{thm}
Let $G$ be a group acting on categories $\calC$ and $\calC'$.
Assume that $\calC$ is $\k$-flat and that the following condition is satisfied:
\begin{itemize}
\item[$(*)$]
There exists a $G$-stable tilting subcategory $E$ of $\Kb(\prj \calC)$
such that there is a $G$-equivariant equivalence $\calC' \to E$.
\end{itemize}
Then the orbit categories $\calC/G$ and $\calC'/G$ are derived equivalent.
\end{thm}
(In the above, $\calC$ is called $\k$-{\em flat} if
all morphism spaces are flat $\k$-modules, and
$E$ is said to be $G$-{\em stable} if the set of objects in $E$ is stable
under the $G$-action on $\Kb(\prj \calC)$ induced from that on $\calC$.)
Observe that if we regard $G$ as a category with a single
object $*$, then a $G$-action on a category $\calC$ is nothing
but a functor $X : G \to \kCat$ with $X(*)=\calC$; and
the orbit category $\calC/G$ coincides with (the $\k$-linear version of) the Grothendieck
construction $\Gr(X)$ of $X$ defined in \cite{Groth}.

The purpose of this paper is to generalize this theorem to
an arbitrary category $I$ and to any {\em colax  functors}\footnote{%%%%%%%%%%
In \cite{Asa-a} we called them {\em oplax} functors.
There are two versions of Grothendieck constructions: (1) for contravariant lax functors
and (2) for covariant colax functors.  Since skew group algebras are formulated as the second
version we deal with colax functors here.  See \cite[Example 2.12] {Asa11}.
}
$X, X' \colon I \to \kCat$
(roughly speaking a colax functor $X$ is a family $(X(i))_{i\in I_0}$ of $\k$-categories indexed by
the objects $i$ of $I$ with an action of $I$,
the precise definition is given in Definition \ref{dfn:colax-fun}).
Recall that if $\calC$ is a category with an action of a group $G$, then
the module category $\Mod \calC$ (resp.\ the derived category $\calD(\Mod \calC)$)
has the induced $G$-action; thus both of them are again categories with $G$-actions.
Hence for a colax functor $X$ the ``module category'' $\Mod X$
(resp.\ the ``derived category'' $\calD(\Mod X)$) should again be a family of categories with
an $I$-action, i.e., a colax functor from $I$ to $\kAb$ (resp.\ to $\kTri$).
In addition, we need a notion of equivalences between colax functors for two purposes:
\begin{enumerate}
\item[(a)] to generalize the statement  $(*)$; and
\item[(b)] to define a derived equivalence of colax functors $X$, $X'$ by an existence of
an equivalence between colax functors  $\calD(\Mod X)$ and $\calD(\Mod X')$.
\end{enumerate}
To define equivalences of objects we need notions of 1-morphisms and 2-morphisms,
thus we need a 2-categorical structure on the collection of colax functors, i.e.,
it is needed to define a 2-category $\Colax(I, \bfC)$ of all colax functors
from $I$ to a 2-category $\bfC$, which can be used for both (a) and (b) above.
Having these things in mind we see that
to generalize the theorem above we have to solve the following problems:

(1) Define the ``module category'' of a colax functor again as
a colax functor.

(2) Define the ``derived category'' of a colax functor as
a colax functor.

(3) Give a natural definition of an equivalence between colax
functors using 2-morphisms of the 2-category of colax functors.

(4) Give a condition on a 1-morphism between colax functors
to be an equivalence. 

(5) Give a natural definition of a derived equivalence between
colax functors by the equivalence (defined in (3)) of
their ``derived categories'' defined in (2).

(6) Characterize the existence of derived equivalences of colax functors
by tilting subcategories, which turns out to be
a generalization of Rickard's Morita theorem for colax functors.

(7) Induce a derived equivalence of Grothendieck constructions
of colax functors from the existence of tilting subcategories,
which will be a generalization of the theorem above.

In our previous paper \cite{Asa-a}
we have solved the problems (1) -- (6) and made clear
the meaning of the condition $(*)$ in the setting of colax functors.
In this paper we solve the problem (7), and in addition we give
a unified way to solve (1) and (2) using the following general statement
on compositions with pseudofunctors (cf.\ Gordon--Power--Street \cite[Subsection 5.6]{GPS95}):

\begin{thm-nn}[Theorem \ref{comp-pseudofun}]
Let $\bfB, \bfC$ and $\bfD$ be $2$-categories and $V \colon \bfC \to \bfD$ a
pseudofunctor.
Then the obvious correspondence $($see subsection \ref{dfn-corr} for details$)$
$$
\Colax(\bfB, V) \colon \Colax(\bfB, \bfC) \to \Colax(\bfB, \bfD)
$$
turns out to be a pseudofunctor.
\end{thm-nn}

The solutions of (1) and (2) use the correspondence  on objects
given by the pseudofunctor $\Colax(\bfB, V)$.
The correspondence on 1-morphisms is needed also to solve (7).
The following is our main result (see Definition \ref{dfn:tilting-colax} for definitions):

\begin{thm-nn}[Theorem \ref{mainthm2}]
Let $X, X' \in \Colax(I, \kCat)$.
Assume that $X$ is $\k$-flat and that there exists a tilting colax functor $\calT$ for $X$
such that $\calT$ and $X'$ are equivalent in $\Colax(I, \kCat)$.
Then $\Gr(X)$ and $\Gr(X')$ are derived equivalent.
\end{thm-nn}

Note that there is an easier way (Lemma \ref{colax-eq}, a solution of (4))
to verify that $\calT$ and $X'$ are equivalent in $\Colax(I, \kCat)$ in the above.
As an easy application, the theorem above gives a unified proof of the
following.

\begin{thm-nn}[Theorem \ref{thm:unified-proof}]
Assume that  $\k$ is a field and that
$\k$-algebras $A$ and $A'$ are derived equivalent.
Then the following pairs are derived equivalent as well:
\begin{enumerate}
\item
path-categories $AQ$ and $A'Q$ for any quiver $Q$;
\item
incidence categories $AS$ and $A'S$ for any poset $S$; and
\item
monoid algebras $AG$ and $A'G$ for any monoid $G$.
\end{enumerate}
\end{thm-nn}
Theorem \ref{mainthm2} can be used to glue many
derived equivalences together as shown in Example \ref{exm:gluing}.

The paper is organized as follows.
In section 2 we recall the definition of the 2-category $\Colax(I, \bfC)$
of colax functors from a category $I$ to a 2-category $\bfC$.
In section 3 we first define a diagonal 2-functor
$\De \colon \kCat \to \Colax(I, \kCat)$ in an obvious way,
and introduce a notion of $I$-coverings $(F, \ps) \colon X \to \De(\calC)$
for a colax functor $X \in \Colax(I, \kCat)_0$ and $\calC \in \kCat_0$
(the subscript 0 stands for objects) as a generalization of
$G$-coverings for a group $G$.
In section 4 we define a $\k$-linear version of Grothendieck construction
as a 2-functor $\Gr \colon \Colax(I, \kCat) \to \kCat$ and introduce
the canonical morphism $(P, \ph) \colon X \to \De(\Gr(X))$.
In section 5 we will show that the Grothendieck construction is a
strict left adjoint to the diagonal 2-functor
with a unit given by the family of canonical morphisms, in particular,
this shows that  the canonical morphism $(P, \ph) \colon X \to \De(\Gr(X))$
is an $I$-covering and any other $I$-covering $X \to \De(\calC)$ is
given as the composite of this followed by $\De(H)$ for an equivalence
$H \colon \Gr(X) \to \calC$. This will be used in the proof of the main result.
In section 6 we redefine the module colax functor $\Mod X \colon I \to \kAb$
and its derived colax functor $\calD(\Mod X) \colon I \to \kTri$
for a colax functor $X \in \Colax(I, \kCat)_0$ by using Theorem \ref{comp-pseudofun}.
In addition, we also define $\Kb(\prj X)$ for $X \in \Colax(I, \kCat)_0$ and
show that this construction preserves $I$-precoverings, which is also used in the proof
of the main result.
It is obvious that the definitions given here coincide
with those given in our previous paper \cite{Asa-a}.
In section 7 we recall the definition of derived equivalences of colax functors in $\Colax(I, \kCat)$
and the theorem characterizing the derived equivalence by tilting colax functors
(Theorem \ref{mainthm1}).
In section 8 we give a proof of Theorem \ref{mainthm2},
and give some applications including an example
of gluing of pieces of derived equivalences together to have a larger one.
In the last section we give a proof of Theorem \ref{comp-pseudofun}.

\section*{Acknowledgements}
Most part of this work was done during my stay in Bielefeld
in February and September, 2010; and a final part (Theorem 6.4)
in September, 2011.
I would like to thank Claus M.\ Ringel and Henning Krause
for their hospitality and nice discussions.
The results were announced at the seminars in the Universities
of Bielefeld, Bonn, Paris 7, and in Beijing Normal University.
I would like to thank Jan Schr\"oer, Bernhard Keller and Changchang Xi
for their kind invitations.
The results were also announced at conferences:
ICRA XIV held in August 2010 in Tokyo (functor version),
the 6-th China-Japan-Korea International Conference on Ring and Module Theory
held in June 2011 at Kyung Hee University at Suwon,
and Shanghai International Conference on Representation
Theory of Algebras held in October 2011 at Shanghai Jiao Tong University.
I would like to thank the organizers for their kind invitations and hospitality.
Finally, I would like to thank D.\ Tamaki for useful discussions with him
on Grothendieck constructions and for his expositions on 2-categorical notions
through his preprints \cite{Tam, Tam2} that aimed at a generalization of \cite{Asa12}.
In addition I would also like to thank the referee for his/her careful reading,
suggestions and questions, by which the paper became easier to read
and I could notice that I forgot to consider the naturality property (0)
of 1-morphisms in Definition \ref{dfn:colax-fun-2cat}
and I could add the verification of this property
in the proof of Lemma \ref{lem:1-morphisms-colax};
also I changed the terminology ``oplax'' to ``colax''.

\section{Preliminaries}

In this section we recall the definition of the 2-category of colax functors
from $I$ to a 2-category from \cite{Asa-a} (see also Tamaki \cite{Tam}).

\begin{dfn}
\label{dfn:colax-fun}
Let $\bfC$ be a 2-category.
A {\em colax functor}
(or an {\em oplax} functor)
from $I$ to $\bfC$ is a triple
$(X, \et, \th)$ of data:
\begin{itemize}
\item
a quiver morphism $X\colon I \to \bfC$, where $I$ and $\bfC$ are regarded as quivers
by forgetting additional data such as 2-morphisms or compositions;
\item
a family $\et:=(\et_i)_{i\in I_0}$ of 2-morphisms $\et_i\colon X(\id_i) \Rightarrow \id_{X(i)}$ in $\bfC$
indexed by $i\in I_0$; and
\item
a family $\th:=(\th_{b,a})_{(b,a)}$ of 2-morphisms
$\th_{b,a} \colon X(ba) \Rightarrow X(b)X(a)$
in $\bfC$ indexed by $(b,a) \in \com(I):=
\{(b,a)\in I_1 \times I_1 \mid ba \text{ is defined}\}$
\end{itemize}
satisfying the axioms:

\begin{enumerate}
\item[(a)]
For each $a\colon i \to j$ in $I$ the following are commutative:
$$
\vcenter{
\xymatrix{
X(a\id_i) \ar@{=>}[r]^(.43){\th_{a,\id_i}} \ar@{=}[rd]& X(a)X(\id_i)
\ar@{=>}[d]^{X(a)\et_i}\\
& X(a)\id_{X(i)}
}}
\qquad\text{and}\qquad
\vcenter{
\xymatrix{
X(\id_j a) \ar@{=>}[r]^(.43){\th_{\id_j,a}} \ar@{=}[rd]& X(\id_j)X(a)
\ar@{=>}[d]^{\et_jX(a)}\\
& \id_{X(j)}X(a)
}}\quad;\text{ and}
$$
\item[(b)]
For each $i \ya{a}j \ya{b} k \ya{c} l$ in $I$ the following is commutative: 
$$
\xymatrix@C=3em{
X(cba) \ar@{=>}[r]^(.43){\th_{c,ba}} \ar@{=>}[d]_{\th_{cb,a}}& X(c)X(ba)
\ar@{=>}[d]^{X(c)\th_{b,a}}\\
X(cb)X(a) \ar@{=>}[r]_(.45){\th_{c,b}X(a)}& X(c)X(b)X(a).
}
$$
\end{enumerate}
\end{dfn}

\begin{dfn}
Let $\bfC$ be a 2-category and $X = (X, \et, \th)$, $X'= (X', \et', \th')$
be colax functors from $I$ to $\bfC$.
A {\em $1$-morphism} (called a {\em left transformation}) from $X$ to $X'$
is a pair $(F, \ps)$ of data
\begin{itemize}
\item
a family $F:=(F(i))_{i\in I_0}$ of 1-morphisms $F(i)\colon X(i) \to X'(i)$
in $\bfC$% indexed by $i \in I_0$
; and
\item
a family $\ps:=(\ps(a))_{a\in I_1}$ of 2-morphisms
$\ps(a)\colon X'(a)F(i) \Rightarrow F(j)X(a)$
$$
\xymatrix{
X(i) & X'(i)\\
X(j) & X'(j)
\ar_{X(a)} "1,1"; "2,1"
\ar^{X'(a)} "1,2"; "2,2"
\ar^{F(i)} "1,1"; "1,2"
\ar_{F(j)} "2,1"; "2,2"
\ar@{=>}_{\ps(a)} "1,2"; "2,1"
}
$$
in $\bfC$ indexed by $a\colon i \to j$ in $I_1$
\end{itemize}
satisfying the axioms
\begin{enumerate}
\item[(a)]
For each $i \in I_0$ the following is commutative:
$$
\vcenter{
\xymatrix{
X'(\id_i)F(i) & F(i)X(\id_i)\\
\id_{X'(i)}F(i) & F(i)\id_{X(i)}
\ar@{=>}^{\ps(\id_i)} "1,1"; "1,2"
\ar@{=} "2,1"; "2,2"
\ar@{=>}_{\et'_iF(i)} "1,1"; "2,1"
\ar@{=>}^{F(i)\et_i} "1,2"; "2,2"
}}\quad;\text{ and}
$$
\item[(b)]
For each $i \ya{a} j \ya{b} k$ in $I$ the following is commutative:
$$
\xymatrix@C=4pc{
X'(ba)F(i) & X'(b)X'(a)F(i) & X'(b)F(j)X(a)\\
F(k)X(ba) & & F(k)X(b)X(a).
\ar@{=>}^{\th'_{b,a}F(i)} "1,1"; "1,2"
\ar@{=>}^{X'(b)\ps(a)} "1,2"; "1,3"
\ar@{=>}_{F(k)\,\th_{b,a}} "2,1"; "2,3"
\ar@{=>}_{\ps(ba)} "1,1"; "2,1"
\ar@{=>}^{\ps(b)X(a)} "1,3"; "2,3"
}
$$
\end{enumerate}
A $1$-morphism $(F, \ps) \colon X \to X'$ is said to be
$I$-{\em equivariant} if $\ps(a)$ is a 2-isomorphism in $\bfC$
for all $a \in I_1$.
\end{dfn}

\begin{dfn}
Let $\bfC$ be a 2-category, $X = (X, \et, \th)$, $X'= (X', \et', \th')$
be colax functors from $I$ to $\bfC$, and
$(F, \ps)$, $(F', \ps')$ 1-morphisms from $X$ to $X'$.
A {\em $2$-morphism} from $(F, \ps)$ to $(F', \ps')$ is a
family $\ze= (\ze(i))_{i\in I_0}$ of 2-morphisms
$\ze(i)\colon F(i) \Rightarrow F'(i)$ in $\bfC$
indexed by $i \in I_0$
such that the following is commutative for all $a\colon i \to j$ in $I$:
$$
\xymatrix@C=4pc{
X'(a)F(i) & X'(a)F'(i)\\
F(j)X(a) & F'(j)X(a).
\ar@{=>}^{X'(a)\ze(i)} "1,1"; "1,2"
\ar@{=>}^{\ze(j)X(a)} "2,1"; "2,2"
\ar@{=>}_{\ps(a)} "1,1"; "2,1"
\ar@{=>}^{\ps'(a)} "1,2"; "2,2"
}
$$
\end{dfn}

\begin{dfn}
Let $\bfC$ be a 2-category, $X = (X, \et, \th)$, $X'= (X', \et', \th')$
and $X''= (X'', \et'', \th'')$
be colax functors from $I$ to $\bfC$, and
let $(F, \ps)\colon X \to X'$, $(F', \ps')\colon X' \to X''$
be 1-morphisms.
Then the composite $(F', \ps')(F, \ps)$ of $(F, \ps)$ and
$(F', \ps')$ is a 1-morphism from $X$ to $X''$ defined by
$$
(F', \ps')(F, \ps):= (F'F, \ps'\circ\ps),
$$
where $F'F:=((F'(i)F(i))_{i\in I_0}$ and for each $a\colon i \to j$ in $I$,
$
(\ps'\circ\ps)(a):= F'(j)\ps(a)\circ \ps'(a)F(i)
$
is the pasting of the diagram
$$
\xymatrix@C=4pc{
X(i) & X'(i) & X''(i)\\
X(j) & X'(j) & X''(j).
\ar_{X(a)} "1,1"; "2,1"
\ar_{X'(a)} "1,2"; "2,2"
\ar^{F(i)} "1,1"; "1,2"
\ar_{F(j)} "2,1"; "2,2"
\ar@{=>}_{\ps(a)} "1,2"; "2,1"
\ar^{X''(a)} "1,3"; "2,3"
\ar^{F'(i)} "1,2"; "1,3"
\ar_{F'(j)} "2,2"; "2,3"
\ar@{=>}_{\ps'(a)} "1,3"; "2,2"
}
$$
\end{dfn}

The following is straightforward to verify.

\begin{prp}
Let $\bfC$ be a $2$-category.
Then colax functors $I \to \bfC$,
$1$-morphisms between them, and $2$-morphisms between
$1$-morphisms $($defined above$)$ define a $2$-category,
which we denote by $\Colax(I, \bfC)$.
\end{prp}

\begin{ntn}\label{ntn:co-op}
Let $\bfC$ be a 2-category.
Then we denote by $\bfC^{\text{op}}$
(resp.\ $\bfC^{\text{co}}$) the 2-category
obtained from $\bfC$ by reversing the 1-morphisms
(resp.\ the 2-morphisms), and we set
$\bfC^{\text{coop}}:=(\bfC^{\text{co}})^{\text{op}}=(\bfC^{\text{op}})^{\text{co}}$.
\end{ntn}

\section{$I$-coverings}

In this section we introduce the notion of $I$-coverings that is a generalization of 
that of $G$-coverings for a group $G$ introduced in \cite{Asa11}, which was
obtained by generalizing the notion of Galois coverings
introduced by Gabriel in \cite{Gab}.
This will be used in the proof of our main theorem.
 
\begin{dfn}
We define a 2-functor $\De\colon \kCat \to \Colax(I, \kCat)$ as follows,
which is called the {\em diagonal} 2-functor:
\begin{itemize}
\item
Let $\calC \in \kCat$. Then $\De(\calC)$ is defined to be a functor
sending each morphism $a\colon i \to j$ in $I$ to
$\id_{\calC}\colon \calC \to \calC$.
\item
Let $E \colon \calC \to \calC'$ be a 1-morphism in $\kCat$.
Then $\De(E)\colon \De(\calC) \to \De(\calC')$ is a 1-morphism
$(F,\ps)$ in $\Colax(I, \kCat)$
defined by $F(i):=E$ and $\ps(a):= \id_E$ for all $i \in I_0$ and all $a \in I_1$:
$$
\xymatrix{
\calC & \calC'\\
\calC & \calC'.
\ar^E "1,1"; "1,2"
\ar^E "2,1"; "2,2"
\ar_{\id_{\calC}} "1,1"; "2,1"
\ar^{\id_{\calC'}} "1,2"; "2,2"
\ar@{=>}_{\id_E}"1,2";"2,1"
}
$$
\item
Let $E, E'\colon \calC \to \calC'$ be 1-morphisms in $\kCat$, and
$\al \colon  E \To E'$ a 2-morphism in $\kCat$.
Then $\De(\al)\colon \De(E) \To \De(E')$ is a 2-morphism in
$\Colax(I, \kCat)$ defined by $\De(\al):= (\al)_{i\in I_0}$.
\end{itemize}
\end{dfn}

\begin{rmk}\label{rmk:1-mor-to-De}
Let $\bfC$ be a 2-category,
$X=(X, \et, \th) \in \Colax(I, \bfC)$, and $C \in \bfC_0$.
Further let
\begin{itemize}
\item $F$ be a family of  1-morphisms $F(i)\colon X(i) \to C$ in $\bfC$
indexed by $i\in I_0$; and
\item $\ps$ be a family 2-morphisms $\ps(a)\colon F(i) \Rightarrow F(j)X(a)$
indexed by $a\colon i \to j$ in $I$:
$$
\xymatrix{
X(i) & C\\
X(j) & C
\ar^{F(i)} "1,1"; "1,2"
\ar_{F(j)} "2,1"; "2,2"
\ar_{X(a)} "1,1";"2,1"
\ar@{=} "1,2";"2,2"
\ar@{=>} "1,2";"2,1"
}
$$
\end{itemize}
Then $(F, \ps)$ is in $\Colax(I, \bfC)(X, \De(C))$
if and only if the following hold.
\begin{enumerate}
\item[(a)]
For each $i \in I_0$ the following is commutative:
$$
\vcenter{
\xymatrix{
F(i) & F(i)X(\id_i)\\
 & F(i)\id_{X(i)}
\ar@{=>}^{\ps(\id_i)} "1,1"; "1,2"
\ar@{=} "1,1"; "2,2"
\ar@{=>}^{F(i)\et_i} "1,2"; "2,2"
}}\quad;\text{ and}
$$
\item[(b)]
For each $i \ya{a} j \ya{b} k$ in $I$ the following is commutative:
$$
\xymatrix@C=4pc{
F(i) &  F(j)X(a)\\
F(k)X(ba) &  F(k)X(b)X(a).
\ar@{=>}^{\ps(a)} "1,1"; "1,2"
\ar@{=>}_{F(k)\,\th_{b,a}} "2,1"; "2,2"
\ar@{=>}_{\ps(ba)} "1,1"; "2,1"
\ar@{=>}^{\ps(b)X(a)} "1,2"; "2,2"
}
$$
\end{enumerate}
\end{rmk}

\begin{dfn}
Let $\calC \in \kCat$ and $(F, \ps) \colon X \to \De(\calC) $ be in $\Colax(I, \kCat)$.
Then 
\begin{enumerate}
\item
$(F, \ps)$ is called an $I$-{\em precovering} (of $\calC$) if the homomorphism
$$
(F,\ps)_{x,y}^{(1)}\colon \Ds_{a\in I(i,j)}X(j)(X(a)x, y) \to \calC(F(i)x, F(j)y)
$$
of $\k$-modules defined by
$(f_a\colon X(a)x \to y)_{a\in I(i,j)} \mapsto \sum_{a\in I(i,j)} F(j)(f_a) \circ\ps(a)(x)$
is an isomorphism for all $i, j \in I_0$ and
all $x \in X(i)_0$, $y \in X(j)_0$.
\item
$(F, \ps)$ is called an $I$-{\em covering} if it is an $I$-precovering and is {\em dense},
i.e., for each $c \in \calC_0$ there exists an $i \in I_0$ and $x \in X(i)_0$
such that $F(i)(x)$ is isomorphic to $c$ in $\calC$.
\end{enumerate}
\end{dfn}

\section{Grothendieck constructions}
In this section we define a 2-functor $\Gr\colon \Colax(I, \kCat) \to \kCat$ whose
correspondence on objects is a $\k$-linear version of (the opposite version of)
the original Grothendieck construction (cf. \cite{Tam}).

\begin{dfn}
We define a 2-functor $\Gr\colon \Colax(I, \kCat) \to \kCat$,
which is called the {\em Grothendieck construction}.

{\bf On objects.}  Let $X=(X, \et, \th) \in \Colax(I, \kCat)_0$.
Then $\Gr(X) \in \kCat_0$ is defined as follows.
\begin{itemize}
\item  $\Gr(X)_0:= \bigcup_{i\in I_0} \{ i \} \times X(i)_0
= \{{}_ix:= (i,x) \mid i \in I_0, x \in X(i)_0\}$.
\item  For each ${}_ix, {}_jy \in \Gr(X)_0$, we set
$$
\Gr(X)({}_ix, {}_jy) := \bigoplus_{a\in I(i,j)} X(j)(X(a)x, y).
$$
\item  For each ${}_ix, {}_jy, {}_kz \in \Gr(X)_0$ and
each $f=(f_a)_{a\in I(i,j)}\in \Gr(X)({}_ix, {}_jy)$,
$g=(g_b)_{b\in I(j,k)}\in \Gr(X)({}_jy, {}_kz)$, we set
$$
g\circ f:= \left(\sum_{\begin{smallmatrix}a\, \in\, I(i,j)\\b\, \in\, I(j,k)\\c\, =\, ba\end{smallmatrix}}
g_b\circ X(b)f_a
\circ \th_{b,a}x \right)_{c\,\in\, I(i,k)},
$$
where each summand is the composite of
$$
X(ba)x \ya{\th_{b,a}x} X(b)X(a)x \ya{X(b)f_a}
X(b)y \ya{g_b} z.$$
\item For each ${}_ix \in \Gr(X)_0$ the identity $\id_{{}_ix}$ is given by
$$
\id_{{}_ix} = (\de_{a,\id_i}\et_i\,x)_{a\in I(i,i)} \in \Ds_{a\in I(i,i)}X(i)(X(a)x,x),
$$
where $\de$ is the Kronecker delta\footnote{%%%%%%%%%%%%%%%%%%%%
This is used to mean that the $a$-th component is $\et_i\,x$ if $a=\id_i$,
and 0 otherwise. 
}.
\end{itemize}
% see [Notes 26, p. 101] for the correspondence on 1-mor and 2-mor

{\bf On 1-morphisms.}
Let $X=(X, \et, \th), X'=(X', \et', \th')$ be objects of $\Colax(I, \kCat)$ and
$(F, \ps)\colon X \to X'$ a 1-morphism in $\Colax(I, \kCat)$.
Then a 1-morphism
$$
\Gr(F, \ps) \colon \Gr(X) \to \Gr(X')
$$
in $\kCat$ is defined as follows.
\begin{itemize}
\item For each ${}_ix \in \Gr(X)_0$, $\Gr(F, \ps)({}_ix):={}_i(F(i)x)$.
\item For each ${}_ix, {}_jy \in \Gr(X)_0$ and
each $f=(f_a)_{a\in I(i,j)} \in \Gr(X)({}_ix, {}_jy)$,
we set
$\Gr(F,\ps)(f):= (F(j)f_a\circ \ps(a)x)_{a\in I(i,j)}$, where each entry is the composite of
$$
X'(a)F(i)x \xrightarrow{\ps(a)x} F(j)X(a)x \xrightarrow{F(j)f_a} F(j)y.
$$
\end{itemize}

{\bf On 2-morphisms.}
Let $X=(X, \et, \th), X'=(X', \et', \th')$ be objects of $\Colax(I, \kCat)$ and
$(F, \ps), (F', \ps') \colon X \to X'$ 1-morphisms in $\Colax(I, \kCat)$,
and let $\ze\colon (F,\ps) \Rightarrow (F', \ps')$ be a 2-morphism in $\Colax(I, \kCat)$.
Then a 2-morphism
$$
\Gr(\ze) \colon \Gr(F,\ps) \Rightarrow \Gr(F', \ps')
$$
in $\kCat$ is defined by
$$
\Gr(\ze){}_ix := (\de_{a, \id_i} \ze(i)x)_{a\in I(i,i)}\colon {}_i(F(i)x)
 \to {}_i(F'(i)x)$$
in $\Gr(X')$ for each ${}_ix \in \Gr(X)_0$.
\end{dfn}

\begin{exm}\label{exm:Gr}
Let $A$ be a $\k$-algebra
regarded as a $\k$-category with a single object.
Then $A \in \kCat_0$.
Consider the functor $X:= \De(A) \colon I \to \kCat$.
Then it is straightforward to verify the following.
\begin{enumerate}
\item If $I$ is a free category defined by the quiver $1 \to 2$,
then $\Gr(X)$ is isomorphic to the triangular algebra
$\bmat{A&0\\A&A}$.
\item If $I$ is a free category defined by a quiver $Q$,
then $\Gr(X)$ is isomorphic to the path-category $AQ$ of $Q$ over $A$.
\item If $I$ is a poset $S$,
then $\Gr(X)$ is isomorphic to the incidence category $AS$ of $S$ over $A$.
\item If $I$ is a monoid $G$,
then $\Gr(X)$ is isomorphic to the monoid algebra\footnote{%%%%%%%%%
Since $AG$ has the identity $1_A1_G$,
this is regarded as a category with a single object.
} $AG$ of $G$ over $A$.
\end{enumerate}
In (3) above, $AS$ is defined to be the factor category
of the path-category $AQ$ modulo the ideal
generated by the full commutativity relations in $Q$,
where $Q$ is the Hasse diagram of $S$ regarded as a quiver by
drawing an arrow $x \to y$ if $x \le y$ in $Q$.
If $S$ is a finite poset, then $AS$ is identified with the usual incidence algebra.

See \cite{Asa-Kim} for further examples of the Grothendieck constructions
of functors, in which the examples (2) and (3) above are unified and generalized.
\end{exm}

\begin{dfn}
Let $X \in \Colax(I, \kCat)$.
We define a left transformation $(P_X, \ph_X):= (P, \ph)\colon X \to \De(\Gr(X))$
(called the {\em canonical morphism}) as follows.
\begin{itemize}
\item For each $i \in I_0$, the functor $P(i)\colon X(i) \to \Gr(X)$ is  defined by
$$
\left\{
\begin{aligned}
P(i)x&:= {}_ix\\
P(i)f &:=(\de_{a,\id_i} f\circ (\et_i\,x))_{a\in I(i,i)}\colon {}_ix \to {}_iy\text{\   in\ $\Gr(X)$}
\end{aligned}
\right.
$$
for all $f\colon x \to y$ in $X(i)$.
\item For each $a \colon i \to j$ in $I$, the natural transformation
$\ph(a)\colon P(i) \Rightarrow P(j)X(a)$
$$\xymatrix{
X(i) & \Gr(X)\\
X(j) & \Gr(X)
\ar^{P(i)} "1,1";"1,2"
\ar_{P(j)} "2,1";"2,2"
\ar_{X(a)} "1,1";"2,1"
\ar@{=} "1,2";"2,2"
\ar@{=>}_{\ph(a)}"1,2";"2,1"
}
$$
is defined by $\ph(a)x:= (\de_{b,a} \id_{X(a)x})_{b \in I(i,j)}$ for all $x \in X(i)_0$.
\end{itemize}
\end{dfn}

\begin{lem}
The $(P, \ph)$ defined above is a $1$-morphism in $\Colax(I, \kCat)$.
\end{lem}

\begin{proof}
This is straightforward by using Remark \ref{rmk:1-mor-to-De}.
\end{proof}

\begin{prp}\label{prp:can-covering}
Let $X \in \Colax(I, \kCat)$.
Then the canonical morphism $(P, \ph)\colon X \to \De(\Gr(X))$
is an $I$-covering.
More precisely, the morphism
$$(P, \ph)_{x,y}^{(1)}\colon \Ds_{a\in I(i,j)}X(j)(X(a)x, y) \to \Gr(X)(P(i)x, P(j)y)$$
is the identity
for all $i, j \in I_0$ and all $x \in X(i)_0$, $y \in X(j)_0$.

\end{prp}

\begin{proof}
By the definitions of $\Gr(X)_0$ and of $P$ it is obvious
that $(P, \ph)$ is dense.
Let $i, j \in I_0$ and $x \in X(i)$, $y \in X(j)$.
We only have to show that
$$(P, \ph)_{x,y}^{(1)}\colon \Ds_{a\in I(i,j)}X(j)(X(a)x, y) \to \Gr(X)(P(i)x, P(j)y)$$
is the identity.
Let $f = (f_a)_{a\in I(i,j)}\in \Ds_{a\in I(i,j)}X(j)(X(a)x, y)$.
Then
$$
\begin{aligned}
(P, \ph)_{x,y}^{(1)}(f) &= \sum_{a\in I(i,j)}P(j)(f_a)\circ \ph(a)x\\
&= \sum_{a\in I(i,j)}(\de_{b,\id_j}f_a\circ (\et_j\,x))_{b\in I(j,j)}\circ
(\de_{c,a}\id_{X(a)x})_{c\in I(i,j)}\\
&= \sum_{a\in I(i,j)}\left(\sum_{\smat{b\in I(j,j)\\c\in I(i,j)\\d=bc}}\de_{b,\id_j}f_a \circ(\et_j\,x)\circ \de_{c,a}\id_{X(b)X(a)x}\circ \th_{b,c}x \right)_{d\in I(i,j)}\\
&= \sum_{a\in I(i,j)}\left(\de_{d,a}f_a \circ (\et_j\,x)\circ \id_{X(\id_j)X(a)x}\circ \th_{\id_j, a}x \right)_{d\in I(i,j)}\\
&= (f_a \circ (\et_j\,x) \circ \th_{\id_j,a}x)_{a\in I(i,j)}
= (f_a)_{a\in I(i,j)}\\
&= f,
\end{aligned}
$$
as required.
\end{proof}

\begin{lem}\label{covering-equivalence}
Let $X \in \Colax(I, \kCat)_0$ and
$H\colon \Gr(X) \to \calC$ be in $\kCat$ and consider the composite $1$-morphism
$(F, \ps) \colon X \ya{(P, \ph)} \De(\Gr(X)) \ya{\De(H)} \De(\calC)$.
Then $(F, \ps)$ is an $I$-covering if and only if $H$ is an equivalence.
\end{lem}

\begin{proof}
Obviously $(F, \ps)$ is dense if and only if so is $H$.
Further for each $i, j \in I_0$, $x \in X(i)$ and $y \in X(j)$,
$(F, \ps)^{(1)}_{x,y}$ is an isomorphism if and only if so is $H_{{}_ix, {}_jy}$
because % $H({}_ix) = F(i)x$ and 
we have a commutative diagram
$$
\xymatrix{
\Ds_{a\in I(i,j)}X(j)(X(a)x, y) & \calC(F(i)x, F(j)y)\\
\Gr(X)({}_ix, {}_jy)
\ar^(.58){(F, \ps)^{(1)}_{x,y}}"1,1";"1,2"
\ar@{=}_{(P,\ph)^{(1)}_{x,y}}"1,1";"2,1"
\ar_{H_{{}_ix, {}_jy}}"2,1";"1,2"
}
$$
by Proposition \ref{prp:can-covering}.
\end{proof}

\section{Adjoints}

In this section we will show that the Grothendieck construction is a strict left adjoint to
the diagonal 2-functor, and that $I$-coverings are essentially given
by the unit of the adjunction.

\begin{dfn}
Let $\calC \in \kCat$.
We define a functor $Q_{\calC} \colon \Gr(\De(\calC)) \to \calC$ by
\begin{itemize}
\item $Q_{\calC}({}_ix) := x$ for all ${}_ix \in \Gr(\De(\calC))_0$; and
\item $Q_{\calC}((f_a)_{a \in I(i,j)}):= \sum_{a\in I(i,j)}f_a$
for all $(f_a)_{a \in I(i,j)} \in \Gr(\De(\calC))({}_ix, {}_jy)$
and for all ${}_ix, {}_jy \in \Gr(\De(\calC))_0$.
\end{itemize}
It is easy to verify that $Q_{\calC}$ is a $\k$-functor.
\end{dfn}

\begin{thm}\label{Gr-De-adjoint}
The $2$-functor $\Gr\colon \Colax(I, \kCat) \to \kCat$ is a strict left $2$-adjoint to
the $2$-functor $\De\colon \kCat \to \Colax(I, \kCat)$.
The unit is given by the family
of canonical morphisms $(P_X, \ph_X) \colon X \to \De(\Gr(X))$
indexed by $X \in \Colax(I, \kCat)$, and the counit is given by the family of
$Q_{\calC} \colon \Gr(\De(\calC)) \to \calC$
defined as above indexed by $\calC \in \kCat$.

In particular, $(P_X, \ph_X)$ has a strict universality in the comma category
$(X \downarrow \De)$, i.e., for each $(F, \ps) \colon X \to \De(\calC)$ in
$\Colax(I, \kCat)$ with $\calC \in \kCat$,
there exists a unique $H \colon \Gr(X) \to \calC $ in $\kCat$ such that
the following is a commutative diagram in $\Colax(I, \kCat)$:
$$
\xymatrix{
X & \De(\calC).\\
\De(\Gr(X))
\ar^{(F,\ps)}"1,1";"1,2"
\ar_{(P_X,\ph_X)} "1,1"; "2,1"
\ar@{-->}_{\De(H)} "2,1";"1,2"
}
$$
\end{thm}

\begin{proof}
For simplicity set $\et:=((P_X, \ph_X))_{X \in \Colax(I, \kCat)_0}$ and
$\ep:= (Q_{\calC})_{\calC \in \kCat_0}$.

\begin{clm}
$\De \ep \cdot \et \De = \id_{\De}$.
\end{clm}
Indeed,
Let $\calC \in \kCat$.
It is enough to show that
$\De(Q_{\calC}) \cdot (P_{\De(\calC)}, \ph_{\De(\calC)}) = \id_{\De(\calC)}$.
Now 
$$
\begin{aligned}
\mathrm{LHS}
 &=\left((Q_{\calC}P_{\De(\calC)}(i))_{i\in I_0}, (Q_{\calC}\ph_{\De(\calC)}(a))_{a\in I_1}\right)
 , \text{ and}\\
\mathrm{RHS}
 &=\left((\id_{\calC})_{i\in I_0}, (\id_{\id_{\calC}})_{a\in I_1}\right).
\end{aligned}
 $$
{\it First entry$\colon$}Let $i \in I$.
Then $Q_{\calC}P_{\De(\calC)}(i) = \id_{\calC}$
because for each $x, y \in \calC_0$ and each $f\in \calC(x, y)$ we have
$(Q_{\calC}P_{\De(\calC)}(i))(x) = Q_{\calC}({}_ix) = x$; and
$(Q_{\calC}P_{\De(\calC)}(i))(f) = (\de_{a, \id_i}f \cdot ((\et_{\De(\calC)})_i\, x))_{a\in I_1} =
\sum_{a \in I(i, i)}\de_{a, \id_i}f = f$.

{\it Second entry$\colon$} Let $a \colon i \to j$ in $I$.
Then $Q_{\calC}\ph_{\De(\calC)}(a) = \id_{\id_{\calC}}$ because
for each $x \in \calC_0$ we have
$Q_{\calC}\left(\ph_{\De(\calC)}(a)x\right)
= Q_{\calC}\left((\de_{b,a}\id_{\De(\calC)(a)x})_{b\in I(i,j)}\right)
=\sum_{b\in I(i,j)}\de_{b,a}\id_x = \id_x = \id_{\id_{\calC}x}$.
This shows that $\mathrm{LHS} = \mathrm{RHS}$.

\begin{clm}
$\ep \Gr \cdot \Gr \et = \id_{\Gr}$.
\end{clm}
Indeed, let $X \in \Colax(I, \kCat)$.
It is enough to show that
$Q_{\Gr(X)}\cdot \Gr(P_X, \ph_X) = \id_{\Gr(X)}$.

{\it On objects$\colon$} Let ${}_ix \in \Gr(X)_0$.
Then $Q_{\Gr(X)}\left(\Gr(P_X, \ph_X)(x)\right)
= Q_{\Gr(X)}({}_i(P_X(i)x))
= {}_ix$.

{\it On morphisms$\colon$} Let $f = (f_a)_{a\in I(i,j)} \colon {}_ix \to {}_jy$ be in $\Gr(X)$.
Then 
$Q_{\Gr(X)}\Gr(P_X, \ph_X)(f)
= Q_{\Gr(X)}((P_X(j)(f_a)\circ\ph_X(a)x)_{a\in I(i,j)})
= \sum_{a\in I(i,j)}P_X(j)(f_a)\circ\ph_X(a)x
= (P_X, \ph_X)^{(1)}_{x,y}(f) = f$.
Thus the claim holds.

The two claims above prove the assertion.
\end{proof}

\begin{cor}\label{covering-Gr}
Let $(F, \ps) \colon X \to \De(\calC)$ be in $\Colax(I, \kCat)$.
Then the following are equivalent.
\begin{enumerate}
\item $(F, \ps)$ is an $I$-covering;
\item There exists an equivalence $H \colon \Gr(X) \to \calC$ such that
the diagram
$$
\xymatrix{
X & \De(\calC)\\
\De(\Gr(X))
\ar^{(F, \ps)}"1,1";"1,2"
\ar_{(P_X, \ph_X)}"1,1";"2,1"
\ar_{\De(H)}"2,1";"1,2"
}
$$
is strictly commutative.
\end{enumerate}
\end{cor}

\begin{proof}
This immediately follows by Theorem \ref{Gr-De-adjoint} and Lemma \ref{covering-equivalence}.
\end{proof}

\section{The Module colax functor}

Let $X\colon I \to \kCat$ be a colax functor.
In this section we simplify the definition of
the ``module category'' $\Mod X$ of $X$ as a colax functor $I \to \kCat$
given in our previous paper \cite{Asa-a}.
Recall that the {\em module category} $\Mod \calC$ of
a category $\calC \in \kCat$ is defined to be the functor category
$\kCat(\calC\op, \Mod \k)$, where $\Mod \k$ denotes the
category of $\k$-modules.
Since $\kCat$ is a 2-category, this is extended to a representable 2-functor
$$
\Mod':= \kCat((\blank)\op, \Mod \k) \colon \kCat \to \kAb\coop
$$
(see Notation \ref{ntn:co-op}).
As is easily seen the composite $\Mod' \circ X$ turns out to be a colax functor
$I \to \kAb\coop$, i.e., a contravariant lax functor $I \to \kAb$.
When $X$ is a group action, namely when $I$ is a group $G$ and $X \colon G \to \kCat$
is a functor, the usual module category $\Mod X$ with a $G$-action of $X$
was defined to be the composite functor
$\Mod X:= \Mod' \circ X \circ i$, where $i \colon G  \to G$ is the group anti-isomorphism
defined by $x \mapsto x\inv$ for all $x \in G$.
In this way we can change $\Mod' \circ X$ to a covariant one.
But in general we cannot assume the existence of such an isomorphism $i$.
Instead in this paper we will use a covariant ``pseudofunctor'' $\Mod \colon \kCat \to \kAb$
defined below and will define $\Mod X$ as the composite $\Mod \circ X$,
which can be seen as a ``lax'' extended version of the module category construction of a category with a $G$-action stated above.
We start with a notion of colax functors from a 2-category to a 2-category.
Compare our definitions of colax functors,
left transformations (1-morphisms) and 2-morphisms in the setting of 2-categories given below
with definitions of morphisms, transformations and modifications in the setting of bicategories
(see Leinster \cite{Lei} for instance).

\begin{dfn}
\label{dfn:colax-fun-2cat}
Let $\bfB$ and $\bfC$ be 2-categories.

(1) A {\em colax functor} from $\bfB$ to $\bfC$ is a triple
$(X, \et, \th)$ of data:
\begin{itemize}
\item
a triple $X=(X_0, X_1, X_2)$ of maps $X_i\colon \bfB_i \to \bfC_i$ ($\bfB_i$ denotes the
collection of $i$-morphisms of $\bfB$ for each $i=0,1,2$) preserving domains and codomains of all 1-morphisms and 2-morphisms
(i.e.\ $X_1(\bfB_1(i,j)) \subseteq \bfC_1(X_0i, X_0j)$
for all $i, j \in \bfB_0$ and $X_2(\bfB_2(a,b)) \subseteq \bfC_2(X_1a, X_1b)$
for all $a, b \in \bfB_1$ (we omit the subscripts of $X$ below));
\item
a family $\et:=(\et_i)_{i\in \bfB_0}$ of 2-morphisms $\et_i\colon X(\id_i) \Rightarrow \id_{X(i)}$ in $\bfC$
indexed by $i\in \bfB_0$; and
\item
a family $\th:=(\th_{b,a})_{(b,a)}$ of 2-morphisms
$\th_{b,a} \colon X(ba) \Rightarrow X(b)X(a)$
in $\bfC$ indexed by $(b,a) \in \com(\bfB):=
\{(b,a)\in \bfB_1 \times \bfB_1 \mid ba \text{ is defined}\}$
\end{itemize}
satisfying the axioms:

\begin{enumerate}
\item[(i)]
$(X_1, X_2) \colon \bfB(i,j) \to \bfC(X_0i,X_0j)$ is a functor
for all $i, j \in \bfB_0$;
\item[(ii)]
For each $a\colon i \to j$ in $\bfB_1$ the following are commutative:
$$
\vcenter{
\xymatrix{
X(a\id_i) \ar@{=>}[r]^(.43){\th_{a,\id_i}} \ar@{=}[rd]& X(a)X(\id_i)
\ar@{=>}[d]^{X(a)\et_i}\\
& X(a)\id_{X(i)}
}}
\qquad\text{and}\qquad
\vcenter{
\xymatrix{
X(\id_j a) \ar@{=>}[r]^(.43){\th_{\id_j,a}} \ar@{=}[rd]& X(\id_j)X(a)
\ar@{=>}[d]^{\et_jX(a)}\\
& \id_{X(j)}X(a)
}}\quad;
$$
\item[(iii)]
For each $i \ya{a}j \ya{b} k \ya{c} l$ in $\bfB_1$ the following is commutative: 
$$
\vcenter{
\xymatrix@C=3em{
X(cba) \ar@{=>}[r]^(.43){\th_{c,ba}} \ar@{=>}[d]_{\th_{cb,a}}& X(c)X(ba)
\ar@{=>}[d]^{X(c)\th_{b,a}}\\
X(cb)X(a) \ar@{=>}[r]_(.45){\th_{c,b}X(a)}& X(c)X(b)X(a)
}}\quad;\text{ and}
$$
\item[(iv)]
For each $a, a' \colon i \to j$ and $b, b' \colon j \to k$ in $\bfB_1$
and each $\al \colon a \to a'$, $\be \colon b \to b'$ in $\bfB_2$
the following is commutative:
$$
\xymatrix{
X(ba) & X(b)X(a)\\
X(b'a') & X(b')X(a').
\ar@{=>}^{\th_{b,a}}"1,1";"1,2"
\ar@{=>}^{\th_{b',a'}}"2,1";"2,2"
\ar@{=>}_{X(\be*\al)}"1,1";"2,1"
\ar@{=>}^{X(\be)*X(\al)}"1,2";"2,2"
}
$$
\end{enumerate}

(2) A {\em lax functor} from $\bfB$ to $\bfC$ is a colax functor
from $\bfB$ to $\bfC^{\text{co}}$ (see Notation \ref{ntn:co-op}).

(3) A {\em pseudofunctor} from $\bfB$ to $\bfC$ is a colax functor with
all $\et_i$ and $\th_{b,a}$ 2-isomorphisms.

(4) We define a 2-category $\Colax(\bfB, \bfC)$ having all the colax functors
$\bfB \to \bfC$ as the objects as follows.

{\bf 1-morphisms.}
Let $X = (X, \et, \th)$, $X'= (X', \et', \th')$
be colax functors from $\bfB$ to $\bfC$.
A {\em $1$-morphism} (called a {\em left transformation}) from $X$ to $X'$
is a pair $(F, \ps)$ of data
\begin{itemize}
\item
a family $F:=(F(i))_{i\in \bfB_0}$ of 1-morphisms $F(i)\colon X(i) \to X'(i)$
in $\bfC$
; and
\item
a family $\ps:=(\ps(a))_{a\in \bfB_1}$ of 2-morphisms
$\ps(a)\colon X'(a)F(i) \Rightarrow F(j)X(a)$
$$\vcenter{\xymatrix{
X(i) & X'(i)\\
X(j) & X'(j)
\ar_{X(a)} "1,1"; "2,1"
\ar^{X'(a)} "1,2"; "2,2"
\ar^{F(i)} "1,1"; "1,2"
\ar_{F(j)} "2,1"; "2,2"
\ar@{=>}_{\ps(a)} "1,2"; "2,1"
}}
$$
in $\bfC$  indexed by $a\colon i \to j \text{ in }\bfB_1$
with the property that
\item[(0)]
for each $\al \colon a \To b$ in $\bfB(i,j)$ the following
is commutative:
\begin{equation}\label{eq:naturality-psi}
\vcenter{\xymatrix@C=10ex{
X'(a)F(i) & X'(b)F(i)\\
F(j)X(a) & F(j)X(b),
\ar@{=>}^{X'(\al)F(i)}"1,1";"1,2"
\ar@{=>}_{F(j)X(\al)}"2,1";"2,2"
\ar@{=>}_{\ps(a)}"1,1";"2,1"
\ar@{=>}^{\ps(b)}"1,2";"2,2"
}}
\end{equation}
thus a family of natural transformations of functors
$$
\vcenter{\xymatrix@C=15ex{
\bfB(i,j) & \bfC(X'(i), X'(j))\\
\bfC(X(i), X(j)) & \bfC(X(i), X'(j))
\ar"1,1";"1,2"^{X'}
\ar"1,1";"2,1"_{X}
\ar"1,2";"2,2"^{\bfC(F(i), X'(j))}
\ar"2,1";"2,2"_{\bfC(X(i), F(j))}
\ar@{=>}"1,2";"2,1"_{\ps_{ij}}
}}\quad(i,j\in \bfB_0)
$$
\end{itemize}
satisfying the axioms
\begin{enumerate}
\item[(a)]
For each $i \in \bfB_0$ the following is commutative:
$$
\vcenter{
\xymatrix{
X'(\id_i)F(i) & F(i)X(\id_i)\\
\id_{X'(i)}F(i) & F(i)\id_{X(i)}
\ar@{=>}^{\ps(\id_i)} "1,1"; "1,2"
\ar@{=} "2,1"; "2,2"
\ar@{=>}_{\et'_iF(i)} "1,1"; "2,1"
\ar@{=>}^{F(i)\et_i} "1,2"; "2,2"
}}\quad;\text{ and}
$$
\item[(b)]
For each $i \ya{a} j \ya{b} k$ in $\bfB_1$ the following is commutative:
$$
\xymatrix@C=4pc{
X'(ba)F(i) & X'(b)X'(a)F(i) & X'(b)F(j)X(a)\\
F(k)X(ba) & & F(k)X(b)X(a).
\ar@{=>}^{\th'_{b,a}F(i)} "1,1"; "1,2"
\ar@{=>}^{X'(b)\ps(a)} "1,2"; "1,3"
\ar@{=>}_{F(k)\,\th_{b,a}} "2,1"; "2,3"
\ar@{=>}_{\ps(ba)} "1,1"; "2,1"
\ar@{=>}^{\ps(b)X(a)} "1,3"; "2,3"
}
$$
\end{enumerate}

{\bf 2-morphisms.}
Let $X = (X, \et, \th)$, $X'= (X', \et', \th')$
be colax functors from $\bfB$ to $\bfC$, and
$(F, \ps)$, $(F', \ps')$ 1-morphisms from $X$ to $X'$.
A {\em $2$-morphism} from $(F, \ps)$ to $(F', \ps')$ is a
family $\ze= (\ze(i))_{i\in \bfB_0}$ of 2-morphisms
$\ze(i)\colon F(i) \Rightarrow F'(i)$ in $\bfC$
indexed by $i \in \bfB_0$
such that the following is commutative for all $a\colon i \to j$ in $\bfB_1$:
$$
\xymatrix@C=4pc{
X'(a)F(i) & X'(a)F'(i)\\
F(j)X(a) & F'(j)X(a).
\ar@{=>}^{X'(a)\ze(i)} "1,1"; "1,2"
\ar@{=>}^{\ze(j)X(a)} "2,1"; "2,2"
\ar@{=>}_{\ps(a)} "1,1"; "2,1"
\ar@{=>}^{\ps'(a)} "1,2"; "2,2"
}
$$

{\bf Composition of 1-morphisms.}
Let $X = (X, \et, \th)$, $X'= (X', \et', \th')$ and $X''= (X'', \et'', \th'')$
be colax functors from $\bfB$ to $\bfC$, and
let $(F, \ps)\colon X \to X'$, $(F', \ps')\colon X' \to X''$
be 1-morphisms.
Then the composite $(F', \ps')(F, \ps)$ of $(F, \ps)$ and
$(F', \ps')$ is a 1-morphism from $X$ to $X''$ defined by
$$
(F', \ps')(F, \ps):= (F'F, \ps'\circ\ps),
$$
where $F'F:=((F'(i)F(i))_{i\in \bfB_0}$ and for each $a\colon i \to j$ in $\bfB$,
$
(\ps'\circ\ps)(a):= F'(j)\ps(a)\circ \ps'(a)F(i)
$
is the pasting of the diagram
$$
\xymatrix@C=4pc{
X(i) & X'(i) & X''(i)\\
X(j) & X'(j) & X''(j).
\ar_{X(a)} "1,1"; "2,1"
\ar_{X'(a)} "1,2"; "2,2"
\ar^{F(i)} "1,1"; "1,2"
\ar_{F(j)} "2,1"; "2,2"
\ar@{=>}_{\ps(a)} "1,2"; "2,1"
\ar^{X''(a)} "1,3"; "2,3"
\ar^{F'(i)} "1,2"; "1,3"
\ar_{F'(j)} "2,2"; "2,3"
\ar@{=>}_{\ps'(a)} "1,3"; "2,2"
}
$$
\end{dfn}

\begin{rmk}
(1) Note that a (strict) 2-functor from $\bfB$ to $\bfC$ is a pseudofunctor with
all $\et_i$ and $\th_{b,a}$ identities.

(2) By regarding the category $I$ as a 2-category with all 2-morphisms identities,
the definition (1) of colax functors above coincides
with Definition \ref{dfn:colax-fun}.

(3) When $\bfB = I$, the definition (4) of $\Colax(\bfB, \bfC)$
above coincides with that of $\Colax(I, \bfC)$ given before.
\end{rmk}

\begin{exm}
\label{exm:Mod-D}
(1) Since $\kCat$ is a 2-category, 
$
\Mod':= \kCat((\blank)\op, \Mod \k) \colon \kCat \to \kAb\coop
$
is a 2-functor, which we can regard as a contravariant lax functor
$$
\Mod':= \kCat((\blank)\op, \Mod \k) \colon \kCat \to \kAb.
$$

(2) We define a pseudofunctor $\Mod \colon \kCat \to \kAb$ as follows.
\begin{itemize}
\item For each $\calC \in \kCat_0$ we set $\Mod \calC:= \Mod' \calC$.
\item For each $F \colon \calC \to \calC'$ in $\kCat_1$ we set
$\Mod F:= \blank\otimes_{\calC}\ovl{F} \colon \Mod\calC \to \Mod\calC'$, where $\ovl{F}$ is the $\calC$-$\calC'$-bimodule
defined by $\ovl{F}(x, y):= \calC'(y, F(x))$ for all $x \in \calC_0$, $y \in \calC'_0$,
which we sometimes write as $\ovl{F}:= \calC'(?, F(\blank))$.
\item For each $\al \colon F \To G$ in $\kCat_2$ (with $F, G \colon \calC \to \calC'$ in $\kCat_1$)
we define $\Mod \al \colon \Mod F \To \Mod G$ by setting
$(\Mod \al)x := \calC'(?, \al x) \colon \calC'(?, Fx) \To \calC'(?, Gx)$ for all $x \in \calC_0$.
\item For each $\calC \in \kCat$ we define
$\et_\calC \colon \Mod \id_\calC \To \id_{\Mod\calC}$
by setting $\et_\calC M \colon M \otimes_\calC \calC(?,\blank) \to M$
to be the canonical isomorphisms for all $M \in \Mod\calC$.
\item For each pair of functors $\calC \ya{F} \calC' \ya{G} \calC''$ in $\kCat$
we define $\th_{G, F} \colon \Mod GF \To \Mod G \circ \Mod F$ as the inverse
of the canonical isomorphism
$$
\blank\otimes_\calC \calC'(?,F(\blank))\otimes_{\calC'} \calC''(?, G(\blank)) 
\To \blank\otimes_\calC \calC''(?, GF(\blank)).
$$
\end{itemize}
It is straightforward to check that this defines a pseudofunctor.

(3) Denote by $\kModCat$ the
2-subcategory of $\kAb$ consisting of the following:
\begin{itemize}
\item objects: $\Mod \calC$ with $\calC \in \kCat_0$,
\item 1-morphisms: functors between objects having exact right adjoints, and
\item 2-morphisms: all natural transformations between 1-morphisms.
\end{itemize}
Then note that the pseudofunctor $\Mod \colon \kCat \to \kAb$ defined above can be seen
as a pseudofunctor $\kCat \to \kModCat$. 
For each $\Mod \calC$ with $\calC \in \kCat_0$ we denote by $\calK_p(\Mod\calC)$
the full subcategory of the homotopy category $\calK(\Mod\calC)$ of $\Mod\calC$ consisting of
{\em homotopically projective} objects $M$, i.e., objects $M$ such that $\calK(\Mod\calC)(M, A) =0$
for all acyclic objects $A$.
Recall that there is a natural embedding
$\bfj_{\calC} \colon \calK_p(\Mod\calC) \to \calD(\Mod\calC)$
having a left adjoint $\bfp_{\calC}$ such that there exists a quasi-isomorphism
$\et_{\calC}M \colon \bfj_{\calC}\bfp_{\calC}M \to M$ for each $M \in \calD(\Mod \calC)$
and that $\bfp_{\calC}\bfj_{\calC} = \id_{\calK_p(\Mod\calC)}$.
Then we can define a pseudofunctor $\calD \colon \kModCat \to \kTri$ as follows.
\begin{itemize}
\item For each $\Mod\calC$ in $\kModCat_0$ with $\calC \in \kCat$ we set
$\calD(\Mod\calC)$ to be the derived category of $\Mod\calC$.
\item For each $F \colon \Mod\calC \to \Mod\calC'$ in $\kModCat_1$,
$F$ naturally induces a functor $\calK F \colon \calK(\Mod\calC) \to \calK(\Mod\calC')$,
which restricts to a functor $\calK_p F \colon \calK_p(\Mod\calC) \to \calK_p(\Mod\calC')$
because $F$ has an exact right adjoint.
Then we set $\calD F$ to be the left derived functor
$\bfL F\colon \calD(\Mod\calC) \to \calD(\Mod\calC')$
of $F$, which is defined as the composite $\bfL F:= \bfj_{\calC'} (\calK_p F)\bfp_{\calC}$.
\item For each $\al \colon F \To F'$ in $\kModCat_2$
with $F, F' \colon \Mod \calC \to \Mod \calC'$ in $\kModCat_1$,
$\al$ naturally induces a natural transformation $\calK_p \al \colon \calK_p F \To \calK_p F'$.
Then we define $\calD \al:= \bfj_{\calC'} (\calK_p \al) \bfp_{\calC}$.
\item We define $\et_{\Mod \calC} \colon \calD(\id_{\Mod \calC}) (=\bfj_{\calC}\bfp_{\calC}) \To \id_{\calD(\Mod \calC)}$
by $\et_{\Mod\calC}:= (\et_{\calC}M)_{M \in \calD(\Mod\calC)}$.
\item Note that for each $\Mod \calC \ya{F} \Mod \calC' \ya{F'} \Mod \calC''$
in $\kModCat_1$
we have $\bfL(F'\circ F) = \bfL F' \circ \bfL F$
because $\bfp_{\calC'}\bfj_{\calC'}= \id_{\calK_p(\Mod \calC)}$.
We define $\th_{F',F} \colon \bfL(F'\circ F) \To \bfL F' \circ \bfL F$ as the identity
$\id_{\bfL(F'\circ F)}$.
\end{itemize}
It is straightforward to check that this defines a pseudofunctor.
\end{exm}

\begin{exm}
\label{exm:prj-Kb}
(1) We define a pseudofunctor $\prj \colon \kCat \to \kadd$
as the subpseudofunctor of $\Mod \colon \kCat \to \kAb \incl \kadd$
by setting $\prj \calC$ to be the full subcategory of $\Mod \calC$
consisting of finitely generated projective $\calC$-modules
for all $\calC \in \kCat_0$,
where $\kadd$ is the full 2-subcategory of $\kCat$ consisting of additive $\k$-categories.
Then for each $F \colon \calC \to \calC'$ in $\kCat_1$
and each $x \in \calC_0$ we have
\begin{equation}\label{prj-representables}
(\prj F)(\calC(\blank, x)) = \calC(\blank, x) \ox_{\calC}\ovl{F} \iso \calC'(\blank, F(x)).
\end{equation}
Note that we can define two 2-functors
$\ds \colon \kCat \to \kadd$ and $\sic \colon \kadd \to \kadd$
by forming formal additive hulls (see e.g., \cite[Subsection 4.1]{Asa99})
and by taking split idempotent completions (see e.g., \cite[Definition 3.1]{Asa11}),
respectively.
Then the Yoneda embeddings $Y_{\calC}\colon \calC \to \prj\calC$,
$x \mapsto \calC(\blank, x)$ ($\calC \in \kCat_0$) induce
a natural 2-isomorphism $Y \colon \sic\circ\ds \To \prj$:
$$
\xymatrix{
\kCat && \kadd.\\
&\kadd
\ar^{\prj}"1,1";"1,3"
\ar_{\ds}"1,1";"2,2"
\ar_{\sic}"2,2";"1,3"
\ar@{=>}_Y^{\iso}"2,2";"1,2"
}
$$

(2) A 2-functor $\Kb \colon \kadd \to \kTri$ is canonically defined by setting
$\Kb(\calC)$ to be the homotopy category of bounded complexes in $\calC$
for all $\calC \in \kadd$.
Then the composite pseudofunctor $\Kb\circ\prj \colon \kCat \to \kTri$
turns out to be a subpseudofunctor of $\calD\circ\Mod \colon \kCat \to \kTri$.
\end{exm}

The following is a useful tool to define new colax functors from an old one by composing with
pseudofunctors.
The proof will be given in the last section.

\begin{thm}\label{comp-pseudofun}
Let $\bfB, \bfC$ and $\bfD$ be $2$-categories and $V \colon \bfC \to \bfD$ a
pseudofunctor.
Then the obvious correspondence $($see subsection \ref{dfn-corr} for details$)$
$$
\Colax(\bfB, V) \colon \Colax(\bfB, \bfC) \to \Colax(\bfB, \bfD)
$$
turns out to be a pseudofunctor.
\end{thm}

\begin{dfn}
Let $X = (X, \et, \th) \in \Colax(I, \kCat)$.

(1) We define the {\em module colax functor}
$\Mod X = (\Mod X, \Mod \et, \Mod \th) \colon I \to \kModCat$ of $X$
as the composite $\Mod X:= \Mod \circ X = \Colax(I, \Mod)(X) \colon I \ya{X} \kCat \ya{\Mod} \kModCat$.
By applying Theorem \ref{comp-pseudofun} to $\bfB:= I$, $\bfC:= \kCat$, $\bfD:= \kModCat$
and $V:= \Mod$ (Example \ref{exm:Mod-D}(2)) we see that $\Mod X \in \Colax(I, \kModCat)$.
Then we have
\begin{itemize}
\item
for each $i \in I_0$,
$(\Mod X)(i) = \Mod (X(i))$; and
\item
for each $a \colon i \to j$ in $I$ the functor
$(\Mod X)(a) \colon (\Mod X)(i) \to (\Mod X)(j)$
is given by $(\Mod X)(a) = \text{-}\otimes_{X(i)}\ovl{X(a)}$,
where $\ovl{X(a)}
$
is an $X(i)$-$X(j)$-bimodule defined by 
$$
\ovl{X(a)}(x, y):= X(j)(y, X(a)(x))
$$
for all $x \in X(i)_0$ and $y \in X(j)_0$.
\end{itemize}

(2) By Theorem \ref{comp-pseudofun} and Example \ref{exm:Mod-D} we can define a colax functor
$\calD(\Mod X) \in \Colax(I, \kTri)$ as the composite $\calD(\Mod X):= \calD \circ \Mod X$,
which we call the {\em derived module colax functor} of $X$.
Then for each $a\colon i \to j$ in $I$,
$\calD(\Mod X)(i) \xrightarrow{\calD(\Mod X)(a)} \calD(\Mod X)(j)$
is equal to
$$\calD(\Mod X(i)) \xrightarrow{\text{-}\Lox_{X(i)}\ovl{X(a)}}\calD(\Mod X(j)).$$

(3) By Theorem \ref{comp-pseudofun} and Example \ref{exm:prj-Kb}
we can define a pseudofunctor
$$
\Colax(I, \Kb\circ\prj) \colon \Colax(I, \kCat) \to \Colax(I, \kTri)
$$
sending each $X \in \Colax(I, \kCat)$ to $\Kb(\prj X)$.
By the remark in Example \ref{exm:prj-Kb}(2) $\Kb(\prj X)$ is
a colax subfunctor of $\calD(\Mod X)$.
\end{dfn}

\begin{rmk}
Let $\calC \in \kCat_0$.
Then it is obvious by definitions that
$$
\De(\Kb(\prj \calC)) = \Kb(\prj \De(\calC)).
$$
\end{rmk}

\begin{prp}
\label{precovering-preserved}
The pseudofunctor $\Kb\circ\prj$ preserves $I$-precoverings, that is,
if $(F, \ps) \colon X \to \De(\calC)$ is an $I$-precovering in $\Colax(I, \kCat)$
with $\calC \in \kCat_0$, then
so is $\Kb(\prj (F, \ps)) \colon \Kb(\prj X) \to \De(\Kb(\prj \calC))$ in $\Colax(I, \kTri)$. 
\end{prp}

\begin{proof}
It is straightforward to verify that the 2-functors $\ds$, $\sic$ and $\Kb$
defined in Example \ref{exm:prj-Kb} preserve $I$-precoverings.
Then the assertion follows from the natural 2-isomorphism $Y \colon \sic\circ\ds \To \prj$.
\end{proof}

\section{Derived equivalences of colax functors}

In this section we recall necessary terminologies and the main theorem
in our previous paper \cite{Asa-a}.
First we cite the following.  See \cite{Asa-a} for the proof.

\begin{lem}
\label{colax-eq}
Let $\bfC$ be a $2$-category and $(F, \ps) \colon X \to X'$
a $1$-morphism in the $2$-category $\Colax(I, \bfC)$.
Then $(F, \ps)$ is an equivalence in $\Colax(I, \bfC)$
if and only if
\begin{enumerate}
\item
For each $i \in I_0$, $F(i)$
% \colon X(i) \to X'(i)$
is an equivalence in $\bfC$; and
\item
For each $a \in I_1$, $\ps(a)$ is a $2$-isomorphism in $\bfC$
$($namely, $(F,\ps)$ is $I$-equivariant$)$.
\end{enumerate}
\end{lem}

\begin{dfn}
Let $X, X' \in \Colax(I, \kCat)$.
Then $X$ and $X'$ are said to be {\em derived equivalent} if
$\calD(\Mod X)$ and $\calD(\Mod X')$ are equivalent
in the 2-category $\Colax(I, \kTri)$.
\end{dfn}

By Lemma \ref{colax-eq} we obtain the following.

\begin{prp}
\label{der-eq-criterion}
Let $X, X' \in \Colax(I, \kCat)$.
Then $X$ and $X'$ are derived equivalent if and only if
there exists a $1$-morphism
$(F, \ps) \colon \calD(\Mod X) \to \calD(\Mod X')$ in $\Colax(I, \kTri)$ such that
\begin{enumerate}
\item
For each $i \in I_0$, $F(i)$
is a triangle equivalence; and
\item
For each $a \in I_1$, $\ps(a)$ is a natural isomorphism
$($i.e., $(F,\ps)$ is $I$-equivariant$)$.
\end{enumerate}
\end{prp}

A $\k$-category $\calA$ is called $\k$-{\em projective}
(resp.\ $\k$-{\em flat}) if
$\calA(x,y)$ are projective (resp.\ flat) $\k$-modules for all $x,y \in \calA_0$.

\begin{dfn}\label{dfn:tilting-colax}
Let $X\colon I \to \kCat$ be a colax functor.
\begin{enumerate}
\item
$X$ is called $\k$-{\em projective} (resp.\ $\k$-{\em flat})
if $X(i)$ are $\k$-projective (resp.\ $\k$-flat) for all $i \in I_0$.
\item
A colax subfunctor $\calT$ of of $\Kb(\prj X)$ is called
{\em tilting} if for each $i \in I_0$,
$\calT(i)$ is a tilting subcategory of $\Kb(\prj X(i))$, namely,
\begin{itemize}
\item
$\Kb(\prj X(i))(U, V[n]) = 0$ for all $U, V \in \calT(i)_0$
and $0 \ne n \in \bbZ$; and
\item
the smallest thick subcategory of $\Kb(\prj X(i))$
containing $\calT(i)$ is equal to $\Kb(\prj X(i))$.
\end{itemize}
\item A tilting colax subfunctor $\calT$ of $\Kb(\prj X)$
with an $I$-equivariant inclusion
$(\si, \ro)\colon$ $\calT \incl \Kb(\prj X)$
is called a {\em tilting colax functor} for $X$.
\end{enumerate}
\end{dfn}

The following was our main theorem in \cite{Asa-a} that gives
a generalization of the Morita type theorem characterizing derived equivalences of categories by Rickard \cite{Rick} and Keller \cite{Ke1} in our setting.

\begin{thm}
\label{mainthm1}
Let $X, X' \in \Colax(I, \kCat)$.
Consider the following conditions.
\begin{enumerate}
\item
$X$ and $X'$ are derived equivalent.
\item
$\Kb(\prj X)$ and $\Kb(\prj X')$ are equivalent
in $\Colax(I, \kTri)$.
\item
There exists a tilting colax functor $\calT$ for $X$
such that $\calT$ and $X'$ are equivalent in $\Colax(I, \kCat)$.
\end{enumerate}
Then 
\begin{enumerate}
\item[(a)]
$(1)$ implies $(2)$.
\item[(b)]
$(2)$ implies $(3)$.
\item[(c)]
If $X'$ is $\k$-projective, then $(3)$ implies $(1)$. 
\end{enumerate}
\end{thm}

\section{Derived equivalences of Grothendieck constructions}

First we cite the statement \cite[Corollary 9.2]{Ke1} in the $\k$-category case.

\begin{thm}[Keller]
Let $\calA$ and $\calB$ be $\k$-categories and assume that
$\calA$ is $\k$-flat.
Then the following are equivalent.
\begin{enumerate}
\item
$\calA$ and $\calB$ are derived equivalent.
\item
$\calB$ is equivalent to a tilting subscategory of $\Kb(\prj \calA)$.
\end{enumerate}

\end{thm}

The following is our main result in this paper.

\begin{thm}
\label{mainthm2}
Let $X, X' \in \Colax(I, \kCat)$.
Assume that $X$ is $\k$-flat and that there exists a tilting colax functor $\calT$ for $X$
such that $\calT$ and $X'$ are equivalent in $\Colax(I, \kCat)$
$($the condition $(3)$ in {\em Theorem \ref{mainthm1}}$)$.
Then $\Gr(X)$ and $\Gr(X')$ are derived equivalent.
\end{thm}

\begin{proof}
Note that $\Gr(X)$ is also $\k$-flat by definition of $\Gr(X)$.
Let $\calT$ be a tilting colax subfunctor of $\Kb(\prj X)$
with an $I$-equivariant inclusion
$(\si, \ro)\colon \calT \incl \Kb(\prj X)$.
Put $(P, \ph):= (P_X, \ph_X)$ for short.
Let $\calT'$ be the full subcategory of $\Kb(\prj \Gr(X))$
consisting of the objects $\Kb(\prj P(i))(U)$ with $i \in I_0$ and $U \in \calT(i)_0$.
Then $\calT'$ is a tilting subcategory of $\Kb(\prj\Gr(X))$.
Indeed, let $L, M \in \calT'_0$ and $0 \ne p \in \bbZ$.
Then $L = \Kb(\prj P(i))(U)$ and $M = \Kb(\prj P(j))(V)$
for some $i, j \in I_0$ and some $U \in \calT(i)_0$, $V \in \calT(j)_0$.
Since
$$
\Kb(\prj(P, \ph)) \colon \Kb(\prj X) \to \De(\Kb(\prj\Gr(X)))
$$
is an $I$-precovering by Proposition \ref{precovering-preserved},
we have
$$
\begin{aligned}
\Kb(\prj\Gr(X))(L, M[p]) &\iso
\Kb(\prj\Gr(X))(\Kb(\prj(P,\ph))(U), \Kb(\prj(P,\ph))(V[p]))\\
&\iso
\Ds_{a\in I(i,j)}\Kb(\prj X(j))(\Kb(\prj X)(a)(U), V[p])\\
&\overset{\rm (a)}{\iso}
\Ds_{a\in I(i,j)}\Kb(\prj X(j))(\calT(a)U, V[p]) \overset{\rm (b)}{=} 0,
\end{aligned}
$$
where the isomorphism (a) follows using the natural isomorphism $\ro(a)$:
$$
\xymatrix{
\calT(i) & \Kb(\prj X(i))\\
\calT(i) & \Kb(\prj X(j))
\ar@{^{(}->}"1,1";"1,2"
\ar@{^{(}->}"2,1";"2,2"
\ar_{\calT(a)}"1,1";"2,1"
\ar^{\Kb(\prj X)(a)}"1,2";"2,2"
\ar@{=>}_{\ro(a)}^{\iso}"1,2";"2,1"
\save "1,1"+<-0.9cm,0.1cm>*\txt{$U\in$}  \restore
}
$$
and the equality (b) follows by assumption from the fact that $\calT(a)U, V \in \calT(j)$.
Now for a triangulated category $\calU$ and a class of objects $\calV$ in $\calU$
we denote by $\thick \calV$ the smallest thick subcategory of $\calU$ containing $\calV$.
Then for each $i \in I_0$ and $x \in X(i)$
we have 
$\Kb(\prj P(i))(X(i)(\blank, x)) \iso (\prj P(i))(X(i)(\blank, x)) \iso \Gr(X)(\blank, P(i)(x))
= \Gr(X)(\blank, {}_ix)$
by the formula \eqref{prj-representables}, and hence
$$
\begin{aligned}%%%%%%%%%%%%%%%%%%%%%%
\Gr(X)(\blank, {}_ix) &\iso
\Kb(\prj P(i))(X(i)(\blank, x))\\
&\in \Kb(\prj P(i))(\thick \calT(i))\\
&\subseteq \thick\{\Kb(\prj P(i))(U) \mid U \in \calT(i)\}\\
&\subseteq \thick \calT'.
\end{aligned}
$$
Therefore $\thick \calT' = \Kb(\prj \Gr(X))$, and hence
$\calT'$ is a tilting subcategory of $\Kb(\prj \Gr(X))$, as desired.
Hence $\Gr(X)$ and $\calT'$ are derived equivalent because $\Gr(X)$
is $\k$-flat.
Let $(F, \ps)$ be the restriction of $\Kb(\prj (P, \ph))$ to $\calT$.
Then $(F, \ps) \colon \calT \to \De(\calT')$ is a dense functor and an $I$-precovering,
thus it is an $I$-covering, which shows that
$\calT' \simeq \Gr(\calT)$ by Corollary \ref{covering-Gr}.
Since $\calT$ and $X'$ are equivalent in $\Colax(I, \kCat)$, we have
$\Gr(\calT) \simeq \Gr(X')$.
As a consequence, $\Gr(X)$ and $\Gr(X')$ are derived equivalent.
\end{proof}

\begin{cor}
\label{der-eq-Gr}
Let $X, X' \in \Colax(I, \kCat)$.
If $X$ and $X'$ are derived equivalent, then
so are $\Gr(X)$ and $\Gr(X')$.
\end{cor}

\begin{proof}
Assume that $X$ and $X'$ are derived equivalent, namely that
the condition (1) in Theorem \ref{mainthm1} is satisfied.
Then the condition (3) in Theorem \ref{mainthm1} holds
by Theorem \ref{mainthm1} (a) and (b).
Hence $\Gr(X)$ and $\Gr(X')$ are derived equivalent
by the theorem above.
\end{proof}

The following is easy to verify.

\begin{lem}
Let $C, C'$ be in $\kCat$.
If $C$ and $C'$ are derived equivalent,
then so are $\De(C)$ and $\De(C')$.\quad\qed
\end{lem}

Corollary \ref{der-eq-Gr} together with the lemma above 
and Example \ref{exm:Gr} gives us
a unified proof of the following fact.

\begin{thm}\label{thm:unified-proof}
Assume that $\k$ is a field and that $\k$-algebras $A$ and $A'$ are derived equivalent.
Then the following pairs are derived equivalent as well:
\begin{enumerate}
\item
path-categories $AQ$ and $A'Q$ for any quiver $Q$;
\item
incidence categories $AS$ and $A'S$ for any poset $S$; and
\item
monoid algebras $AG$ and $A'G$ for any monoid $G$. 
\end{enumerate}
\qed
\end{thm}

\begin{exm}%example of gluing many small derived equivalences together to have very large one
\label{exm:gluing}
Assume that $\k$ is a field.
Let $n$ be a natural number $\ge 3$, and $I$ the free category defined by the quiver $Q$:
$2 \ya{a_2} 3 \ya{a_3} \cdots \ya{a_{n-1}} n$.
Define functors $X, X' \colon I \to \kCat$ as follows.

For each $i \in I_0 =\{2, \dots, n\}$ let $X(i)$ be the $\k$-category defined by the quiver
$$
\xymatrix{
1 & 2 & 3 & \cdots & i
\ar@/^/^{\al_1}"1,1";"1,2"
\ar@/^/^{\al_2}"1,2";"1,3"
\ar@/^/^{\al_3}"1,3";"1,4"
\ar@/^/^{\al_{i-1}}"1,4";"1,5"
\ar@/^/^{\be_1}"1,2";"1,1"
\ar@/^/^{\be_2}"1,3";"1,2"
\ar@/^/^{\be_3}"1,4";"1,3"
\ar@/^/^{\be_{i-1}}"1,5";"1,4"
}
$$
with relations
$\al_{j+1}\al_{j}=0$, $\be_j\be_{j+1}=0$, $\al_j\be_j = \be_{j+1}\al_{j+1}$ for all $j = 1,\dots, i-1$
and $\al_1\be_1\al_1 = 0$, $\be_{i-1}\al_{i-1}\be_{i-1}=0$.
For each $a_i\colon i \to i+1$ in $I_1$ let $X(a_i) \colon X(i) \to X(i+1)$ be the
inclusion functor. This defines a functor $X\colon I \to \kCat$.

For each $i \in I_0 =\{2, \dots, n\}$ let $X'(i)$ be the $\k$-category defined by the quiver
$$
\xymatrix{
1 & 2 & 3 & \cdots & i
\ar^{\ga_1}"1,1";"1,2"
\ar^{\ga_2}"1,2";"1,3"
\ar^{\ga_3}"1,3";"1,4"
\ar^{\ga_{i-1}}"1,4";"1,5"
\ar@/^15pt/^{\ga_i}"1,5";"1,1"
}
$$
with relations
$\ga_{j+i}\cdots\ga_{j+1}\ga_{j} =0$ for all $j \in \bbZ/i\bbZ$.
For each $a_i\colon i \to i+1$ in $I_1$ let $X(a_i) \colon X(i) \to X(i+1)$ be the
functor defined by the correspondence $1 \mapsto 1$, $j \mapsto j+1$ and $\al_1 \mapsto \al_2\al_1$, $\al_j \mapsto \al_{j+1}$ for all $j=2, \dots i$.
This defines a functor $X'\colon I \to \kCat$.

As is explained in \cite{Asa97} we have a tilting spectroid $\calT(i)$ for $X(i)$
that is a full subcategory of $\Kb(\prj X(i))$ consisting of the following $i$ objects
$$
\begin{aligned}
T(i)_1&:= (\udl{P_1}),\\
T(i)_2&:= (\udl{P_2}\ya{P(\al_2)}P_3\ya{P(\al_3)}\cdots \ya{P(\al_{i-1})}P_i),\\
T(i)_3&:= (\udl{P_2}\ya{P(\al_2)}P_3\ya{P(\al_3)}\cdots \ya{P(\al_{i-2})}P_{i-1}),\\
&\vdots\\
T(i)_i&:=(\udl{P_2}),
\end{aligned}
$$
where $P_j:= X(i)(\blank, j) \in \prj X(i)$ for all $j \in X(i)_0$, $P(\al):= X(i)(\blank, \al)$
for all $\al \in X(i)_1$ and the underline indicates the place of degree zero.
Again by \cite{Asa97}, $\calT(i)$ is presented by the same quiver with relations
as $X'(i)$ and we have an isomorphism
$F(i)\colon X'(i) \to \calT(i)$ sending $j$ to $T(i)_j$ for all $j=1,\dots, i$ and
$\ga_j$ to a morphism $\de(i)_j\colon T(i)_j \to T(i)_{j+1}$ for all $j \in \bbZ/i\bbZ$,
where $\de(i)_1:=(\udl{P(\al_1)})$, $\de(i)_j:=(\udl{\id_{P_2}}, \dots, \id_{P_{i-j+1}},0)$
for all $j=2,\dots,i-1$ and $\de(i)_{i}:=(\udl{P(\be_1}))$.
Thus $\calT(i)$ gives a derived equivalence between $X(i)$ and $X'(i)$.

For each $a_i \colon i \to i+1$ in $I_1$ define a functor $\calT(a_i) \colon \calT(i) \to \calT(i+1)$
by the correspondence $T(i)_1 \mapsto T(i+1)_1$,
$T(i)_j \mapsto T(i+1)_{j+1}$ and $\de(i)_1 \mapsto \de(i+1)_2\de(i+1)_1$, $\de(i)_j \mapsto \de(i+1)_{j+1}$ for all $j =2,\dots, i$.
This defines a functor $\calT \colon I \to \kCat$.
Then we have a strict commutative diagram
$$
\xymatrix{
X'(i) & \calT(i)\\
X'(i+1) & \calT(i+1)
\ar^{F(i)}"1,1";"1,2"
\ar_{F(i+1)}"2,1";"2,2"
\ar_{X'(a_i)}"1,1";"2,1"
\ar^{\calT(a_i)}"1,2";"2,2"
}
$$
in $\kCat$ for all $i \in I_0$, which shows that
$X'$ and $\calT$ are equivalent in $\Colax(I, \kCat)$.
Finally by definition of $\calT(a_i)$'s it is easy to see that we have an $I$-equivariant inclusion
$(\si, \ro)\colon \calT \incl \Kb(\prj X)$:
$$
\xymatrix{
\calT(i) & \Kb(\prj X(i))\\
\calT(i+1) & \Kb(\prj X(i+1)).
\ar@{^{(}->}^-{\si(i)}"1,1";"1,2"
\ar@{^{(}->}_-{\si(i+1)}"2,1";"2,2"
\ar_{\calT(a_i)}"1,1";"2,1"
\ar^{\Kb(\prj X(a_i))}"1,2";"2,2"
\ar@{=>}_{\rho(a_i)}^{\iso}"1,2";"2,1"
}
$$
Hence by Theorem \ref{mainthm2} we can glue derived equivalences between
$X(i)$'s and $X'(i)$'s together to have a derived equivalence
between $\Gr(X)$ and $\Gr(X')$.
For example when $n = 5$, these are presented by the following quivers
$$
\Gr(X)=
\vcenter{
\xymatrix{
1 & 2\\
1 & 2 & 3\\
1 & 2 & 3 & 4\\
1 & 2 & 3 & 4 & 5
\ar@/^/^{\al_1}"1,1";"1,2"
\ar@/^/^{\be_1}"1,2";"1,1"
\ar@/^/^{\al_1}"2,1";"2,2"
\ar@/^/^{\al_2}"2,2";"2,3"
\ar@/^/^{\be_1}"2,2";"2,1"
\ar@/^/^{\be_2}"2,3";"2,2"
\ar@/^/^{\al_1}"3,1";"3,2"
\ar@/^/^{\al_2}"3,2";"3,3"
\ar@/^/^{\al_3}"3,3";"3,4"
\ar@/^/^{\be_1}"3,2";"3,1"
\ar@/^/^{\be_2}"3,3";"3,2"
\ar@/^/^{\be_3}"3,4";"3,3"
\ar@/^/^{\al_1}"4,1";"4,2"
\ar@/^/^{\al_2}"4,2";"4,3"
\ar@/^/^{\al_3}"4,3";"4,4"
\ar@/^/^{\al_{4}}"4,4";"4,5"
\ar@/^/^{\be_1}"4,2";"4,1"
\ar@/^/^{\be_2}"4,3";"4,2"
\ar@/^/^{\be_3}"4,4";"4,3"
\ar@/^/^{\be_{4}}"4,5";"4,4"
\ar"1,1";"2,1"
\ar"1,2";"2,2"
\ar"2,1";"3,1"
\ar"2,2";"3,2"
\ar"2,3";"3,3"
\ar"3,1";"4,1"
\ar"3,2";"4,2"
\ar"3,3";"4,3"
\ar"3,4";"4,4"
}
},
\quad
\Gr(X') =
\vcenter{
\xymatrix{
1 & 2\\
1 & 2 & 3\\
1 & 2 & 3 & 4\\
1 & 2 & 3 & 4 & 5
\ar^{\ga_1}"1,1";"1,2"
\ar@/^15pt/^{\ga_2}"1,2";"1,1"
\ar^{\ga_1}"2,1";"2,2"
\ar^{\ga_2}"2,2";"2,3"
\ar@/^15pt/^{\ga_3}"2,3";"2,1"
\ar^{\ga_1}"3,1";"3,2"
\ar^{\ga_2}"3,2";"3,3"
\ar^{\ga_3}"3,3";"3,4"
\ar@/^15pt/^(.7){\ga_4}"3,4";"3,1"
\ar^{\ga_1}"4,1";"4,2"
\ar^{\ga_2}"4,2";"4,3"
\ar^{\ga_3}"4,3";"4,4"
\ar^{\ga_{4}}"4,4";"4,5"
\ar@/^15pt/^{\ga_5}"4,5";"4,1"
\ar"1,1";"2,1"
\ar"1,2";"2,3"
\ar"2,1";"3,1"
\ar"2,2";"3,3"
\ar"2,3";"3,4"
\ar"3,1";"4,1"
\ar"3,2";"4,3"
\ar"3,3";"4,4"
\ar"3,4";"4,5"
}
}
$$
with suitable relations as calculated in \cite{Asa-Kim}.
Note that if we start with $I$ presented by the same quiver $Q$ as above with
relations $a_{i+1}a_i=0$ for all $i=2,\dots, n-2$, then
both $\Gr(X)$ and $\Gr(X')$ are presented by the same quivers
with relations consisting of the same relations as before together with
the additional relations that the vertical paths of length 2 are zero, respectively.
\end{exm}

\section{The composite of colax functors and pseudofunctors}

In this section we prove Theorem \ref{comp-pseudofun}.
%See Notes34, p.18 for the summary of the proof.
Throughout this section $\bfB, \bfC$ and $\bfD$ are $2$-categories.

\begin{ntn}
When we denote a colax functor by a letter $X$ the 1-st (resp.\ 2-nd and 3-rd) entry of $X$
is denoted by $X_{012}:=(X_0, X_1, X_2)$ (resp.\ $\et^X$ and $\th^X$), thus we set
$X = (X_{012}, \et^X, \th^X)$, and sometimes we simply write $X$ for $X_{d}$ for all $d = 0,1,2$
if this seems to make no confusion.
\end{ntn}

\subsection{Correspondences on cells}
\label{dfn-corr}

\begin{lem}
Let  $X\colon \bfB \to \bfC$ and $V \colon \bfC \to \bfD$ be colax functors.
We define the composite $VX \colon \bfB \to \bfD$ as follows.
\begin{itemize}
\item $(VX)_d:= V_d X_d \colon \bfB_d \ya{X_d} \bfC_d \ya{V_d} \bfD_d$ for all $d = 0,1,2$.
\item $\et^{VX}_i:= \et^V_{X(i)} \circ V\et^X_i
\colon
\xymatrix{
VX(\id_i) &V(\id_{X(i)}) &\id_{(VX)(i)}
\ar@{=>}^{V\et^X_i}"1,1";"1,2"
\ar@{=>}^{\et^V_{X(i)}}"1,2";"1,3"
}$
for all $i \in \bfB_0$.
\item $\th^{VX}_{b,a}:= \th^V_{X(b), X(a)}\circ V\th^X_{b,a} \colon
\xymatrix@C=35pt{VX(ba) & V(X(b)\circ X(a)) &VX(b) \circ VX(a)
\ar@{=>}^-{V\th^X_{b,a}}"1,1";"1,2"
\ar@{=>}^{\th^V_{X(b), X(a)}}"1,2";"1,3"
}
$
for all $(b,a) \in \com(\bfB)$.
\end{itemize}
Namely, $VX:=((V_0X_0, V_1X_1, V_2X_2), (\et^V_{X(i)} \circ V\et^X_i)_{i\in \bfB_0},
(\th^V_{X(b), X(a)}\circ V\th^X_{b,a})_{(b,a)\in \com(\bfB)})$.
Then the composite $\Colax(\bfB, V)(X):= VX \colon \bfB \to \bfD$ is again a colax functor.
\end{lem}

\begin{proof}
It is enough to verify the axioms (i) -- (iv) in Definition \ref{dfn:colax-fun-2cat}.

(i) $((VX)_1, (VX)_2)\colon \bfB(i,j) \ya{(X_1, X_2)}\bfC(X(i),X(j))
\ya{(V_1, V_2)} \bfD(VX(i), VX(j))$ is a functor for all $i, j \in \bfB_0$
as a composite of the functors
$(X_1, X_2)$ and $(V_1, V_2)$.

(ii) For each $a \colon i \to j$ in $\bfB$ we have the following commutative diagram:
$$
\xymatrix@C=45pt{
VX(a)\id_{VX(i)} & VX(a) V(\id_{X(i)}) & VX(a) VX(\id_i)\\
 &    V(X(a)\id_{X(i)}) & V(X(a)X(\id_i))\\
 && VX(a\id_i).
 \ar@{=>}_{VX(a)\et^V_{X(i)}}"1,2";"1,1"
 \ar@{=>}_{VX(a)V(\et^X_{i})}"1,3";"1,2"
 \ar@{=>}_{V(X(a)\et^X_{i})}"2,3";"2,2"
 \ar@{=>}_{\th^V_{X(a), \id_{X(i)}}}"2,2";"1,2"
 \ar@{=>}_{\th^V_{X(a), X(\id_{i})}}"2,3";"1,3"
 \ar@{=>}_{V(\th^X_{a, \id_{i}})}"3,3";"2,3"
 \ar@{=}"2,2";"1,1"
 \ar@{=}"3,3";"2,2"
}
$$
The commutativity of the square follows from the axiom (iv) for $\th^V$.
The remaining commutative diagram is obtained similarly.
These two commutative diagrams verify the axiom (ii) of colax functors.

(iii) For each $i \ya{a} j \ya{b} k \ya{c} l$ in $\bfB$ we have the following commutative diagram:
$$
\xymatrix@C=45pt{
VX(cba) & V(X(c)X(ba)) & VX(c)\cdot VX(ba)\\
V(X(cb)X(a)) & V(X(c)X(b)X(a)) & VX(c) V(X(b)X(a))\\
VX(cb)\cdot VX(a) & V(X(c)X(b)) VX(a) & VX(c)\cdot VX(b) \cdot VX(a),
% horizontal arrows
\ar@{=>}^{V\th^X_{c, ba}}"1,1";"1,2"
\ar@{=>}^{\th^V_{X(c),X(ba)}}"1,2";"1,3"
\ar@{=>}^{V(\th^X_{c,b}\id_{X(a)})}"2,1";"2,2"
\ar@{=>}^{\th^V_{X(c),X(b)X(a)}}"2,2";"2,3"
\ar@{=>}_{V(\th^X_{c,b})VX(a)}"3,1";"3,2"
\ar@{=>}_{\th^V_{X(c), X(b)} VX(a)}"3,2";"3,3"
% vertical arrows
\ar@{=>}_{V\th^X_{cb,a}}"1,1";"2,1"
\ar@{=>}^{V(\id_{X(c)}\th^X_{b,a})}"1,2";"2,2"
\ar@{=>}^{VX(c)\cdot V\th^X_{b,a}}"1,3";"2,3"
\ar@{=>}_{\th^V_{X(cb),X(a)}}"2,1";"3,1"
\ar@{=>}^{\th^V_{X(c)X(b), X(a)}}"2,2";"3,2"
\ar@{=>}^{VX(c)\th^V_{X(b),X(a)}}"2,3";"3,3"
}
$$
which verifies the axiom (iii) of colax functors.

(iv) Let $a, a' \colon i \to j$; $b, b' \colon j \to k$;
$\al \colon a \To a'$ and $\be \colon b \To b'$ be in $\bfB$.
Then we have the following commutative diagram:
$$
\xymatrix@C=40pt{
VX(ba) & V(X(b)\cdot X(a)) & VX(b)\cdot VX(a)\\
VX(b'a') & V(X(b')\cdot V(a')) & VX(b')\cdot VX(a'),
% horizontal arrows
\ar@{=>}^-{V(\th^X_{b,a})}"1,1";"1,2"
\ar@{=>}^{\th^V_{X(b),X(a)}}"1,2";"1,3"
\ar@{=>}_-{V(\th^X_{b',a'})}"2,1";"2,2"
\ar@{=>}_{\th^V_{X(b'),X(a')}}"2,2";"2,3"
% vertical arrows
\ar@{=>}_{VX(\be\cdot \al)}"1,1";"2,1"
\ar@{=>}^{V(X\be \cdot X\al)}"1,2";"2,2"
\ar@{=>}^{VX\be \cdot VX\al}"1,3";"2,3"
}
$$
which verifies the axiom (iv) of colax functors.
\end{proof}

\begin{lem}\label{lem:1-morphisms-colax}
Let $X, X' \colon \bfB \to \bfC$ and $V \colon \bfC \to \bfD$ be colax functors and
$(F,\ps) \colon X \to X'$ a $1$-morphism in $\Colax(\bfB, \bfC)$, and
consider the diagram
\begin{equation}
\label{induced-1-morphism}
\vcenter{
\xymatrix@C=9em@R=9em{
VX(i) & VX'(i)\\
VX(j)  & VX'(j).
% horizontal arrows
\ar^{VF(i)} "1,1";"1,2" 
\ar_{VF(j)} "2,1";"2,2" 
% vertical arrows
\ar_{VX(a)} "1,1";"2,1" 
\ar^{VX'(a)} "1,2";"2,2"
% diagonal arrows
\ar@/^2.5pc/|{V(X'(a)F(i))} "1,1";"2,2" \ar@/_2.5pc/|{V(F(j)X(a))} "1,1";"2,2"
% double arrows
\ar@{=>}^-{\th^V_{X'(a),F(i)}} "1,2"+<-2.5em,-2.5em>;"1,2"
\ar@{=>}^{V\ps(a)} "1,2"+<-6em,-4.7em>;"2,1"+<5.2em,4.5em> 
\ar@{=>}^-{\th^V_{F(j),X(a)}} "2,1"+<2.5em,2.5em> ;"2,1"
}}
\end{equation}
Assume that $\th^V_{d,c}$ are isomorphisms for all $(d,c) \in \com(\bfC)$
$($e.g., that $V$ is a pseudofunctor$)$.
Then we can define a $1$-morphism $\Colax(\bfB, V)(F, \ps):= V(F,\ps) \colon VX \to VX'$
in $\Colax(\bfB, \bfD)$ by
$$\begin{aligned}
V(F,\ps):&=((V(F(i)))_{i\in \bfB_0}, (\ps_V(a))_{a \in \bfB_1}),
\text{ where for }a:i \to j\\
\ps_V(a):&= \th^V_{F(j), X(a)}\cdot V(\ps(a)) \cdot \th^V_{X'(a), F(i)}{}\inv.
\end{aligned}
$$
\end{lem}

\begin{proof}
We set $X = (X, \et, \th)$ and $X' = (X',  \et', \th')$ for short.

First, the functor $V_{12} \colon \bfC(X(i), X'(j)) \to \bfD(VX(i), VX'(j))$
sends the commutative square \eqref{eq:naturality-psi}
to the commutative square $(*)$ below
{\footnotesize
$$\vcenter{\xymatrix@R=4.5ex@C=8ex{
VX'(a)\cdot VF(i) &&&VX'(b)\cdot VF(i)\\
&V(X'(a)F(i) )& V(X'(b)F(i))&\\
&V(F(j)X(a)) & V(F(j)X(b))\\
VF(j)\cdot VX(a) &&& VF(j)\cdot VX(b),
%% inner square
\ar@{=>}^{V(X'(\al)F(i))}"2,2";"2,3"
\ar@{=>}_{V(F(j)X(\al))}"3,2";"3,3"
\ar@{=>}_{V(\ps(a))}"2,2";"3,2"
\ar@{=>}^{V(\ps(b))}"2,3";"3,3"
\ar@{}|{(*)}"2,2";"3,3"
%% outer square
\ar@{=>}^{VX'(\al)*VF(i)}"1,1";"1,4"
\ar@{=>}_{VF(j)*VX(\al)}"4,1";"4,4"
\ar@{=>}_{\ps_V(a)}"1,1";"4,1"
\ar@{=>}^{\ps_V(b)}"1,4";"4,4"
%% diagonal arrows
\ar@{=>}_{\th^V_{X'(a),F(i)}}"2,2";"1,1"
\ar@{=>}^(0.4){\th^V_{F(j),X(a)}}"3,2";"4,1"
\ar@{=>}^{\th^V_{X'(b),F(i)}}"2,3";"1,4"
\ar@{=>}^(0.4){\th^V_{F(j),X(b)}}"3,3";"4,4"
%% (iv) %%
\save "1,2"+<11.5ex,-2.5ex>*{\text{\tiny (iv)}} \restore
\save "4,2"+<11.5ex,2.5ex>*{\text{\tiny (iv)}} \restore
%% dfn %%
\save "2,1"+<9ex,-4ex>*{\text{\tiny (definition)}} \restore
\save "2,4"+<-9ex,-4ex>*{\text{\tiny (definition)}} \restore
}}
$$
}\noindent
which is completed to the commutative diagram above.
Hence the family $(\ps_V(a))_{a \in \bfB_1}$ has the property (0)
of 1-morphisms in $\Colax(\bfB, \bfD)$ (Definition \ref{dfn:colax-fun-2cat}(4)).

(a) For each $i \in \bfB_0$ we have the following commutative diagram:
{\footnotesize
$$
\xymatrix@C=3.5em{
VX'(\id_i)\cdot VF(i) & V(X'(\id_i)\cdot F(i)) & V(F(i)\cdot X(\id_i)) & VF(i)\cdot VX(\id_i)\\
V(\id_{X'(i)})\cdot VF(i) &V(\id_{X'(i)}\cdot F(i)) & V(F(i)\cdot \id_{X(i)}) & VF(i)\cdot V(\id_{X(i)})\\
\id_{VX'(i)}\cdot VF(i) &&& VF(i)\cdot \id_{VX(i)},
% horizontal arrows
\ar@{=>}_{\th^V_{X'(\id_i),F(i)}}"1,2";"1,1"
\ar@{=>}^{V\ps(\id_i)}"1,2";"1,3"
\ar@{=>}^{\th^V_{F(i),X(\id_i)}}"1,3";"1,4"
\ar@{=>}^-{\th^V_{\id_{X'(i)},F(i)}}"2,2";"2,1"
\ar@{=}"2,2";"2,3"
\ar@{=>}_-{\th^V_{F(i),\id_{X(i)}}}"2,3";"2,4"
\ar@{=}"3,1";"3,4"
% vertical arrows
\ar@{=>}^{V\et'_i\cdot VF(i)}"1,1";"2,1"
\ar@{=>}^{V(\et'_i\id_{F(i)})}"1,2";"2,2"
\ar@{=>}^{V(\id_{F(i)}\et_i)}"1,3";"2,3"
\ar@{=>}_{VF(i)\cdot V\et_i}"1,4";"2,4"
\ar@{=>}^{\et^V_{X'(i)}\cdot VF(i)}"2,1";"3,1"
\ar@{=>}_{VF(i)\cdot \et^V_{X(i)}}"2,4";"3,4"
}
$$
}
which verifies the axiom (a) of 1-morphisms.

(b) For each $i \ya{a} j \ya{b} k$ in $\bfB$ we have the following commutative diagrams:
{\tiny
$$
\xymatrix{
VX'(ba)\cdot VF(i) & V(X'(b)X'(a))\cdot VF(i) & VX'(b)\cdot VX'(a)\cdot VF(i)\\
&& VX'(b)V(X'(a)F(i)) &VX'(b)\cdot V(F(j)X(a))\\
&&& VX'(b)\cdot VF(j)\cdot VX(a)\\
V(X'(ba)F(i)) & V(X'(b)X'(a)F(i)) & V(X'(b)F(j)X(a)) & V(X'(b)F(j))VX(a)
% horizontal
\ar@{=>}^{V(\th'_{b,a})VF(i)}"1,1";"1,2"
\ar@{=>}^{\th^V_{X'(b),X'(a)}VF(i)}"1,2";"1,3"
\ar@{=>}^{V(\id_{X'(b)})V(\ps(a))}"2,3";"2,4"
\ar@{=>}^-{V(\th'_{b,a}F(i))}"4,1";"4,2"
\ar@{=>}^{V(X'(b)\ps(a))}"4,2";"4,3"
\ar@{=>}^{\th^V_{X'(b)F'(j),X(a)}}"4,3";"4,4"
% vertical
\ar@{=>}_{VX'(b)\th^V_{X'(a),F(i)}{}\inv}"1,3";"2,3"
\ar@{=>}^{VX'(b)\cdot \th^V_{F(j),X(a),}}"2,4";"3,4"
\ar@{=>}^{\th^V_{X'(b),F(j)}{\inv VX(a)}}"3,4";"4,4"
\ar@{=>}^{\th^V_{X'(ba),F(i)}}"4,1";"1,1"
\ar@{=>}^{\th^V_{X'(b)X'(a),F(i)}}"4,2";"1,2"
% diagonal
\ar@{=>}^{\th^V_{X(b),X'(a)F(i)}}"4,2";"2,3"
\ar@{=>}^{\th^V_{X'(b),F(j)X(a)}}"4,3";"2,4"
}
$$
}
and
{\tiny
$$
\xymatrix{
V(X'(ba)F(i)) & V(X'(b)X'(a)F(i)) & V(X'(b)F(j)X(a)) & V(X'(b)F(j))VX(a)\\
V(F(k)X(ba)) &   &  V(F(k)X(b)X(a)) & V(F(k)X(b))VX(a)\\
VF(k)\cdot VX(ba) && VF(k)\cdot V(X(b)X(a)) & VF(k)\cdot VX(b) \cdot VX(a).
% horizontal
\ar@{=>}^-{V(\th'_{b,a}F(i))}"1,1";"1,2"
\ar@{=>}^{V(X'(b)\ps(a))}"1,2";"1,3"
\ar@{=>}^{\th^V_{X'(b)F'(j),X(a)}}"1,3";"1,4"
\ar@{=>}^{V(F(k)\th_{b,a})}"2,1";"2,3"
\ar@{=>}^{\th^V_{X(b)F'(j),X(a)}}"2,3";"2,4"
\ar@{=>}"1,3";"1,4"
\ar@{=>}_{VF(k)\cdot V\th_{b,a}}"3,1";"3,3"
\ar@{=>}_{VF(k)\cdot \th^V_{X(b),X(a)}}"3,3";"3,4"
% vertical
\ar@{=>}_{V\ps(ba)}"1,1";"2,1"
\ar@{=>}_{V(\ps(b)X(a))}"1,3";"2,3"
\ar@{=>}^{V(\ps(b))VX(a)}"1,4";"2,4"
\ar@{=>}_{\th^V_{F(k),X(ba)}}"2,1";"3,1"
\ar@{=>}_{\th^V_{F(k),X(b)X(a)}}"2,3";"3,3"
\ar@{=>}^{\th^V_{F(k),X(b)}VX(a)}"2,4";"3,4"
}
$$
}
Glue these two diagrams together along the common row to get a large diagram,
which verifies the axiom (b) of 1-morphisms.
\end{proof}

\begin{lem}
Let $X, X' \colon \bfB \to \bfC$ and $V \colon \bfC \to \bfD$ be colax functors,
$(F,\ps), (F', \ps') \colon X \to X'$ $1$-morphisms, and
$\al \colon (F, \ps) \To (F', \ps')$ a $2$-morphism in $\Colax(\bfB, \bfC)$.
Assume that all $\th^V_{d,c}$ are isomorphisms $($e.g., that $V$ is a pseudofunctor$)$.
Then we can define a $2$-morphism $\Colax(\bfB, V)(\al):= V\al \colon V(F,\ps) \To V(F', \ps')$ 
in $\Colax(\bfB, \bfD)$ by
$$
V\al := (V\al_i)_{i\in \bfB_0}.
$$
\end{lem}

\begin{proof}
Let $a\colon i \to j$ be in $\bfB$.
It is enough to show the commutativity of the following diagram:
$$
\xymatrix@C=3em{
VX'(a)\cdot VF(i) & V(X'(a)F(i)) & V(F(j)X(a)) & VF(j)\cdot VX(a)\\
VX'(a)\cdot VF'(i) & V(X'(a)F'(i)) & V(F'(j)X(a)) & VF'(j)\cdot VX(a).
% horizontal
\ar@{=>}_{\th^V_{X'(a),F(i)}}"1,2";"1,1"
\ar@{=>}^{V(\ps(a))}"1,2";"1,3"
\ar@{=>}^{\th^V_{F(j),X(a)}}"1,3";"1,4"
\ar@{=>}^{\th^V_{X'(a),F'(i)}}"2,2";"2,1"
\ar@{=>}_{V(\ps'(a))}"2,2";"2,3"
\ar@{=>}_{\th^V_{F'(j),X(a)}}"2,3";"2,4"
% vertical
\ar@{=>}_{VX'(a)\cdot V\al_i}"1,1";"2,1"
\ar@{=>}_{V(X'(a)\al_i)}"1,2";"2,2"
\ar@{=>}^{V(\al_jX(a))}"1,3";"2,3"
\ar@{=>}^{V\al_i\cdot VX(a)}"1,4";"2,4"
}
$$
Since $\al = (\al_i \colon F(i) \To F'(i))_{i\in \bfB_0}$ is a 2-morphism in $\Colax(\bfB, \bfC)$,
we have the commutative diagram
$$
\xymatrix{
X'(a)F(i) & F(j)X(a)\\
X'(a)F'(i) & F'(j)X(a).
\ar@{=>}^{\ps(a)}"1,1";"1,2"
\ar@{=>}_{\ps'(a)}"2,1";"2,2"
\ar@{=>}_{X'(a)\al_i}"1,1";"2,1"
\ar@{=>}^{\al_jX(a)}"1,2";"2,2"
}
$$
This gives the commutativity of the central square of the diagram above
by applying the functor $(V_1, V_2)$ to it.
The axiom (iv) of colax functors for $V$ shows the commutativity of the remaining squares.
\end{proof}

\subsection{Proof of Theorem \ref{comp-pseudofun}}
By the three lemmas above we can define a correspondence
$$
\Colax(\bfB, V)_{012} \colon \Colax(\bfB, \bfC) \to \Colax(\bfB, \bfD)
$$
sending $i$-cells to $i$-cells for all $i = 0, 1, 2$ preserving domains and codomains.
It remains to define families $H = (H_X)_{X\in \Colax(\bfB, \bfC)_0}$ and
$\Th = (\Th_{F',F})_{(F',F) \in \com(\Colax(\bfB, \bfC))}$
and to show that $\Colax(\bfB, V):= (\Colax(\bfB, V)_{012}, H, \Th)$ becomes a pseudofunctor
$\Colax(\bfB, \bfC) \to \Colax(\bfB, \bfD)$.

For each $X\in \Colax(\bfB, \bfC)_0$ we define $H_X \colon V(\id_X) \To \id_{VX}$ by setting
$$H_X:= (\et_{X(i)}^V \colon V(\id_{X(i)}) \to \id_{VX(i)})_{i\in \bfB_0}.$$
Then $H_X$ turns out to be a 2-morphism
because by definitions of $\th^V$ and $\et^V$ we have a commutative diagram
$$
\xymatrix{
VX(a)\cdot V(\id_{X(i)} & V(X(a)\cdot \id_{X(i)}) & V(\id_{X(j)}X(a)) & V(\id_{X(j)})\cdot VX(a)\\
VX(a)\cdot \id_{VX(i)} &&& \id_{VX(j)}VX(a)
% horizontal
\ar@{=>}^{(\th_{X(a),\id_{X(i)}}^V)\inv}"1,1";"1,2"
\ar@{=}"1,2";"1,3"
\ar@{=>}^{\th_{\id_{X(j)},X(a)}^V}"1,3";"1,4"
\ar@{=}"2,1";"2,4"
% vertical
\ar@{=>}^{VX(a)\cdot \et_{X(i)}^V}"1,1";"2,1"
\ar@{=>}_{\et_{X(j)}^V\cdot VX(a)}"1,4";"2,4"
}
$$
for all $a \colon i \to j$ in $\bfB$.
Note that $H_X$ are isomorphisms because $\et_k^V$ are for all $k \in \bfC_0$.

For each $(F',F) \in \com(\Colax(\bfB, \bfC))$, say
$F \colon X \To X'$ and $F' \colon X' \To X''$,
we define $\Th_{F',F} \colon V(F'F) \To VF'\circ VF$ by setting
$$
\Th_{F',F}:= (\th_{F'(i), F(i)}^V \colon V(F'(i)F(i)) \to VF'(i)\cdot VF(i))_{i\in \bfB_0}.
$$
Then $\Th_{F',F}$ turns out to be a 2-morphism.
Indeed, it is enough to show the commutativity of the diagram
$$
\xymatrix{
VX''(a)\cdot V(F'(i)F(i)) & V(F'(j)F(j))\cdot VX(a)\\
VX''(a)\cdot VF'(i) \cdot VF(i) & VF'(j)VF(j) VX(a)
% horizontal
\ar@{=>}^{\Ps(a)}"1,1";"1,2"
\ar@{=>}_{\Ps'(a)}"2,1";"2,2"
% vertical
\ar@{=>}_{VX''(a)\cdot \th_{F'(i), F(i)}^V}"1,1";"2,1"
\ar@{=>}^{\th_{F'(j),F(j)}^V\cdot VX(a)}"1,2";"2,2"
}
$$
for all $a \colon i \to j$ in $\bfB$, where we set
$V(F'F) = ((V(F'(i)F(i))_{i\in \bfB_0}, (\Ps(a))_{a\in \bfB_1})$
and $VF'\cdot VF = ((VF(i)\cdot VF(i))_{i\in \bfB_0}, (\Ps'(a))_{a\in \bfB_1})$,
namely
$$
{\footnotesize
\begin{aligned}
\Ps(a) &:= \th_{F'(j)F(j),X(a)}^V\cdot V((F'(j)\cdot \ps(a)) \cdot V(\ps'(a)\cdot F(i))
\cdot ({\th'}_{X''(a), F'(i)F(i)}^V)\inv\\
\Ps'(a) &:= (VF'(j) \cdot (\th_{F(j),X(a)}^V\cdot V\ps(a)\cdot ({\th'}_{X'(a),F(i)}^V)\inv) \circ
(\th_{F'(j),X'(a)}^V\cdot V\ps'(a)\cdot ({\th'}_{X''(a),F'(i)}^V)\inv\cdot VF(i))
\end{aligned}
}$$
for all $a \colon i \to j$ in $\bfB$.
This follows from the coassociativity of $V$ and the naturality of $\th^V$.
% diagram in [2011-09-30-W5] p.6
%$$
%\xymatrix{
%VX''(a)V(F'(i)F(i)) & V(X''(a)F'(i)F(i)) & V(F'(j)X'(a)F(i))\\
% & V(X''(a)F'(i))VF(i) & V(F'(j)X'(a))VF(i)\\
% VX''(a)VF'(i)VF(i) && VF'(j)VX'(a)VF(i)
%\ar@{=>}"1,1";"1,2"
%\ar@{=>}"1,2";"1,3"
%\ar@{=>}"2,2";"2,3"
%\ar@{=>}"3,1";"3,3"
%%
%\ar@{=>}"1,1";"3,1"
%\ar@{=>}"1,2";"2,2"
%\ar@{=>}"1,3";"2,3"
%\ar@{=>}"2,2";"3,1"
%\ar@{=>}"2,3";"3,3"
%}
%$$
%and
%$$
%\xymatrix{
%V(F'(j)X'(a)F(i)) & V(F'(j)F(j)X(a)) & V(F'(j)F(j))VX(a)\\
%}
%$$
Note that $\Th_{F',F}$ are isomorphisms because $\th_{b,a}^V$ are for all $a, b\in \bfC_0$.

Now the defining conditions of $\th^V$ and $\et^V$ directly show
that $(\Colax(\bfB, V)_{012}, H, \Th)$ is a colax functor, hence
a pseudofunctor because all $H_X$ and $\Th_{F',F}$ are isomorphisms.
\qed

\end{document}